\documentclass[12pt]{amsart}
\usepackage{amssymb,amsfonts,amsmath,amsthm,cite,verbatim}
\usepackage{xcolor}
\usepackage{tikz}
\usepackage{graphicx}
\usepackage[left=2.5cm, right=2.5cm, top=3.5cm, bottom=3.5cm]{geometry}
\usepackage{graphicx}
\usepackage{hyperref}
\usepackage{mathdots}
\usepackage{diagbox}
\usepackage[normalem]{ulem}
\newtheorem{theorem}{Theorem}[section]
\newtheorem{theo}[theorem]{Theorem}
\newtheorem{prop}[theorem]{Proposition}
\newtheorem{lem}[theorem]{Lemma}
\newtheorem{rem}[theorem]{Remark}
\newtheorem{defn}[theorem]{Definition}
\newtheorem{cor}[theorem]{Corollary}

\newtheorem{exam}[theorem]{Example}

\makeatletter \@addtoreset{equation}{section}

\DeclareMathOperator*{\CT}{CT}

\def\x{\mathbf{x}}
\def\y{\mathbf{y}}

\title{Magic labelling enumeration on pseudo-line graphs and pseudo-cycle graphs}

\author{Guoce Xin}

\address{School of Mathematical Sciences, Capital Normal University,
 Beijing 100048, P.R. China}

\email{guoce\_xin@163.com}

\author{Yueming Zhong*}

\thanks{*Corresponding author}

\address{School of Mathematics and Statistics, Hainan University, Haikou 570228,
P.R. China}

\email{zhongyueming107@gmail.com}

\author{Yangbiao Zhou}

\address{School of Mathematics and Statistics, Hainan University, Haikou 570228,
P.R. China}

\email{yangbiao\_zhou@163.com}

\date{\today}

\begin{document}

\begin{abstract}
Stanley's theorem establishes that for any finite graph $G$, the number $h_G(s)$ of magic labelings with magic sum $s$ can be expressed as a sum of two polynomials in $s$. However, determining the precise form of $h_G(s)$ is generally challenging. This paper aims to compute $h_G(s)$ and its generating function for pseudo-line graphs and pseudo-cycle graphs, thereby extending the earlier work of B\'{o}na et al.\cite{Bona-1,Bona}.
\end{abstract}

\maketitle

\vspace{-5mm}

\noindent
\begin{small}
 \emph{Mathematic subject classification}: Primary 05A19; Secondary 11D04; 05C78.
\end{small}

\noindent
\begin{small}
\end{small}

\section{Introduction}\label{sec-intro}
Throughout this paper, we adopt standard notation: $\mathbb{R}$ (real numbers), $\mathbb{Q}$ (rational numbers), $\mathbb{Z}$ (integers), $\mathbb{N}$ (nonnegative integers), and $\mathbb{P}$ (positive integers).

Let $G = (V, E)$ be a finite undirected graph with vertex set $V = \{v_1, \ldots, v_m\}$ and edge set $E = \{a_1, \ldots, a_n\}$. A \textit{magic labelling} assigns nonnegative integer labels to the edges such that the sum of labels incident to each vertex equals a constant $s$ (the \textit{magic sum} or \textit{index}). Formally, for a labelling function $\mu: E \to \mathbb{N}$:
\begin{equation}\label{e-magic-sum}
\text{wt}(v_i) := \sum_{\substack{(v_i, v_j) \in E \\ j = 1}}^{m} \mu(v_i, v_j) = s, \quad \forall i \in \{1, \ldots, m\}. \tag{1.1}
\end{equation}
Magic labellings were first studied for simple graphs in \cite{MacDougall}. Vertex magic total labellings correspond to our framework with self-looped vertices. For related work, see \cite{Matthias1,Matthias2,Arnold,Kotzig,Baker2,Prasanna,xin-zhong,Xin-cube,MacDougall,Wallis}. (Note: Terminology ``magic'' varies across the literature.)

Let $h_G(s)$ denote the number of magic labellings of $G$ with magic sum $s$. Stanley's seminal result states:
\begin{theo}[Stanley, \cite{Stanley-magiclabelings}]\label{theo-Stanley}
For any finite graph $G$, there exist polynomials $\varphi_G(s)$ and $\psi_G(s)$ such that:
\[
h_G(s) = \varphi_G(s) + (-1)^s \psi_G(s).
\]
If the loop-free subgraph of $G$ is bipartite, then $\psi_G(s) = 0$, making $h_G(s)$ polynomial.
\end{theo}

Theorem \ref{theo-Stanley} guarantees that $h_G(s)$ can be expressed as a sum of two polynomials. However, determining the explicit forms of $\varphi_G(s)$ and $\psi_G(s)$, as well as their generating functions, is generally difficult for arbitrary graphs. In this paper, we focus on two specific families of graphs: \textit{pseudo-line graphs} and \textit{pseudo-cycle graphs}. Our aim is to compute their magic sum counting function $h_G(s)$ and to determine the corresponding generating functions.

To present our results, we need the following definitions. Let $\mathbf{k} = (k_1,\ldots,k_n) \in \mathbb{N}^n$ be a vector of nonnegative integers.
\subsection{Definitions}
\begin{itemize}
    \item $L_{n,m}$: A pseudo-line graph on $n$ vertices with $m$ self-loops per vertex (Fig. \ref{fig:Lnm}).
    \item $C_{n,\mathbf{k}}$: A circular graph on $n$ vertices where the $i$-th vertex $i$ has $k_i$ self-loops.
    \item $C_{n,m}$: A pseudo-cycle graph on $n$ vertices with $m$ self-loops per vertex (Fig. \ref{fig:Cnm}).
\end{itemize}

\begin{figure}[!ht]
\centering{
\includegraphics[height=1.1 in]{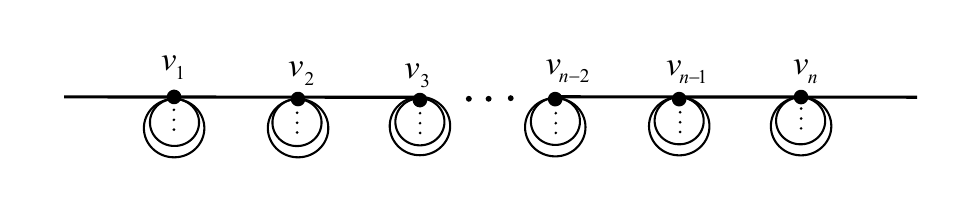}}\vspace{-0.5cm}
\caption{The pseudo-line graph $L_{n,m}$.}
\label{fig:Lnm}
\end{figure}

\begin{figure}[!ht]
\centering{
\includegraphics[height=2.5 in]{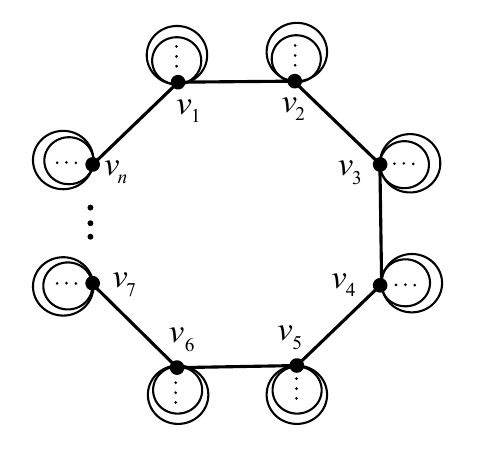}}\vspace{-0.5cm}
\caption{The pseudo-cycle graph $C_{n,m}$.}
\label{fig:Cnm}
\end{figure}

Specific cases of pseudo-cycle graphs $C_{1,m}$ and $C_{2,m}$ are illustrated below:
\begin{figure}[!ht]
\centering{
\includegraphics[height=1.0 in]{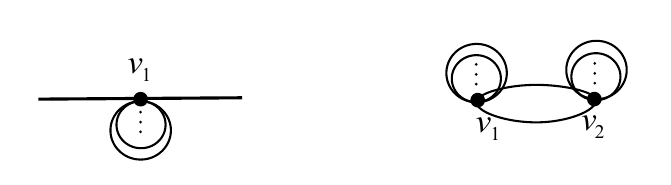}}\vspace{-0.5cm}
\caption{The pseudo-cycle graphs $C_{1,m}$ and $C_{2,m}$.}
\label{fig:Cnm12}
\end{figure}

\begin{rem}
In the labelling of the pseudo-cycle graph $C_{1,m}$, the edges incident to the vertex $v_1$ are subjected to specific constraints: the two non-loop edges are assigned the same nonnegative integer. For the case $C_{1,2}$, the counting function $h_{C_{1,2}}(s)$ enumerates the number of non-negative integer solutions $(x_1, x_2)$ to the linear inequality derived from the magic sum condition:
\begin{equation}\label{eq-c12-sum}
x_1 + x_2 + x_1 \le s \quad \iff \quad 2x_1 + x_2 \le s.
\end{equation}
We can verify this for small values of $s$:
\begin{itemize}
\item If $s=2$, the valid solutions $(x_1, x_2)$ are $(0,0), (0,1), (0,2)$, and $(1,0)$, yielding $h_{C_{1,2}}(2) = 4$.
\item If $s=3$, the set of solutions is augmented by $(0,3)$ and $(1,1)$, leading to $h_{C_{1,2}}(3) = 6$.
\end{itemize}
\end{rem}

We define $FL_{m}(s, y)$ and $FC_{m}(s, y)$ as the generating functions that enumerate the magic labelings with sum $s$ for the graphs $L_{n,m}$ and $C_{n,m}$, respectively:
\[
FL_{m}(s, y) = \sum_{n=0}^{\infty} h_{L_{n,m}}(s) y^n, \quad FC_{m}(s, y) = \sum_{n=0}^{\infty} h_{C_{n,m}}(s) y^n.
\]

These initial values are set to $h_{L_{0,m}}(s) = h_{C_{0,m}}(s) = s+1$. This convention ensures the transfer-matrix method, as established in Proposition \ref{prop-L}, remains valid for $n=0$. For instance, in the case $m=2$, we have the representations $h_{L_{n,2}}(s) = \mathbf{u}_{s+1}^\top B_{s+1}^n \mathbf{u}_{s+1}$ and $h{C_{n,2}}(s) = \operatorname{tr}(B_{s+1}^n)$, where $B_{s+1}^0$ is interpreted as the identity matrix $I_{s+1}$.

As a consistency check, we note that for $m=0$ (i.e., graphs with no loops), the following simple closed forms can be verified directly:
\[
h_{L_{n,0}}(s)=s+1,
\qquad
h_{C_{n,0}}(s) = \begin{cases}
    s+1, & \text{if } n \text{ is even},\\[4pt]
    \dfrac{1 + (-1)^s}{2}, & \text{if } n \text{ is odd}.
\end{cases}
\]
Consequently, we obtain the generating functions
\begin{align*}
FL_{0}(s, y) &= \frac{s + 1}{1 - y}=s + 1 + (s + 1)y + (s + 1)y^2 + (s + 1)y^3 + O(y^4),\\
FC_{0}(s, y) &= \frac{1 + (-1)^s}{2} \cdot \frac{y}{1-y^2} + \frac{s + 1}{1 - y^2} \\
&= s + 1 + \left(\frac{1 + (-1)^s}{2}\right)y + (s + 1)y^2 + \left(\frac{1 + (-1)^s}{2}\right)y^3 + O(y^4).
\end{align*}

\subsection{Main Results}
We begin our results by extending the theory of the generating functions $FL_{m}(s, y)$ and $FC_{m}(s, y)$. Our work builds upon prior results for $m=1$, which include closed forms for $\frac{FL_{1}(s, y) - 1}{y}$ and $FC_{1}(s, y)$ \cite{Bona-1}, recursive formulas for $FL_{1}(s, y)$ \cite{Bona}, and independent verification \cite{xin-zhong}. Our first main contribution is the generalization of these results to the case $m=2$, as shown in Theorems \ref{theo-1} and \ref{theo-2}.

We will utilize two classes of polynomials, $P_n(y)$ and $Q_n(y)$, defined recursively as follows.
\begin{align*}
P_n(y) &= -y[P_{n-1}(-y) + P_{n-3}(-y)] + 2P_{n-2}(y) - P_{n-4}(y), \\
Q_n(y) &= -y[Q_{n-1}(-y) + Q_{n-3}(-y)] + 2Q_{n-2}(y) - Q_{n-4}(y),
\end{align*}
satisfying the initial conditions
\begin{align*}
P_0(y) &= 1, & Q_0(y) &= 1 - y, \\
P_1(y) &= 2, & Q_1(y) &= 1 - 2y - y^2, \\
P_2(y) &= 3 - 2y, & Q_2(y) &= 1 - 4y - 2y^2 + y^3, \\
P_3(y) &= 4 - 4y - 2y^2, & Q_3(y) &= 1 - 6y - 7y^2 + 2y^3 + y^4.
\end{align*}

\begin{theorem}\label{theo-1}
For $s \in \mathbb{P}$,
\[
FL_{2}(s, y) = \frac{P_s(y)}{Q_s(y)}.
\]
\end{theorem}

\begin{theorem}\label{theo-2}
For $s \in \mathbb{P}$,
\[
FC_{2}(s, y) = - \frac{y\frac{d}{dy}Q_s(y)}{Q_s(y)}+s+1 .
\]
\end{theorem}

Our second result characterizes $h_{C_{n,\mathbf{k}}}(s)$ for arbitrary vectors $\mathbf{k} = (k_1, \ldots, k_n) \in \mathbb{P}^n$.
\begin{theorem}\label{theo-hck}
For $\mathbf{k} = (k_1, \ldots, k_n) \in \mathbb{P}^n$ and $s \in \mathbb{P}$:
\[
h_{C_{n,\mathbf{k}}}(s) = \varphi_{n,\mathbf{k}}(s) + (-1)^s \frac{1 + (-1)^{n+1}}{2^{\sum_{i=1}^{n}k_i + 2}},
\]
where $\varphi_{n,\mathbf{k}}(s)$ is a degree-$\sum_{i=1}^n k_i$ polynomial.
\end{theorem}

\subsection{Structure}
Section \ref{sec-2} establishes Theorems \ref{theo-1} and \ref{theo-2} using transfer matrix methods. Section \ref{sec-3} proves Theorem \ref{theo-hck} via polytope decomposition and the constant term method. Section \ref{sec-4} analyzes the generating functions $h_{C_{n,2}}(s)$ and $h_{L_{n,2}}(s)$ with respect to the magic sum $s$.

\section{Proof of Theorems \ref{theo-1} and \ref{theo-2}}\label{sec-2}
In this section, we prove Theorems \ref{theo-1} and \ref{theo-2} using the transfer matrix method and the matrix algebraic computation method.

\subsection{Transfer matrix method}
This subsection is devoted to determining the numbers $h_{C_{n,2}}(s)$ and $h_{L_{n,2}}(s)$ of magic labelings by constructing a transfer matrix.

The numbers of magic labelings for the pseudo-line graph $h_{L_{n,2}}(s)$ and the pseudo-cycle graph $h_{C_{n,2}}(s)$ are given by the number of nonnegative integer solutions to the following systems of linear inequalities, respectively.
\begin{equation}\label{eq-CLn}
(1)\begin{cases}
   x_1     + x_2     + x_3  \leq  s, \\
   x_3     + x_4     + x_5  \leq  s, \\
         \vdots            \\
   x_{2n-1} + x_{2n} + x_{2n+1}  \leq  s, \\
\end{cases}\quad
(2)\begin{cases}
   x_1     + x_2     + x_3  \leq  s, \\
   x_3     + x_4     + x_5  \leq  s, \\
         \vdots\\
   x_{2n-1} + x_{2n} + x_1  \leq  s. \\
\end{cases}
\end{equation}

Let $a_{ij}(n)$ count the nonnegative integer solutions to system (1) in \eqref{eq-CLn}
with boundary conditions $x_1=i$ and $x_{2n+1}=j$.
The magic labelling counts become
\begin{equation*}
h_{L_{n,2}}(s)=\sum_{i=0}^{s}\sum_{j=0}^{s}a_{ij}(n),\quad
h_{C_{n,2}}(s)=\sum_{i=0}^{s}a_{ii}(n).
\end{equation*}

We define the base transfer matrix $B_{s+1}=(a_{ij}(1))$ $(i,j=0,\dots,s)$ to be the $(s+1)\times (s+1)$ matrix with entries $a_{ij}(1)$ given by
\[
a_{ij}(1)=
\begin{cases}
s+1-i-j, & \text{if }i+j \leq s,\\
0, & \text{otherwise}.
\end{cases}
\]

\begin{exam}For $n=5,6$,
\[
B_5=
\begin{pmatrix}
 5&    4&    3&    2&   1\\
 4&    3&    2&    1&   0\\
 3&    2&    1&    0&   0\\
 2&    1&    0&    0&   0\\
 1&    0&    0&    0&   0
\end{pmatrix}
,\quad
B_6=
\begin{pmatrix}
 6&    5&    4&    3&   2&  1\\
 5&    4&    3&    2&   1&  0\\
 4&    3&    2&    1&   0&  0\\
 3&    2&    1&    0&   0&  0\\
 2&    1&    0&    0&   0&  0\\
 1&    0&    0&    0&   0&  0
\end{pmatrix}.
\]
\end{exam}

Let $\mathbf{u}_n = (1,1,\dots,1)^\top \in \mathbb{R}^n$ be the all-ones vector. Using the transfer matrix method, we obtain the following proposition.
\begin{prop}\label{prop-L}
For $n,s \in \mathbb{P}$,
\begin{equation}\label{eq-hLC}
h_{L_{n,2}}(s)=\mathbf{u}_{s+1}^\top B_{s+1}^n\mathbf{u}_{s+1}
,\quad
h_{C_{n,2}}(s)=\mathrm{trace}(B_{s+1}^n).
\end{equation}
\end{prop}

From the matrix expressions in \eqref{eq-hLC}, we derive the following generating functions.
\begin{lem}\label{eq-FCL2}
For $n,s \in \mathbb{P}$, the generating functions are given by
\[
FL_{2}(s,y)=\sum\nolimits_{n=0}^{+\infty}\mathbf{u}_{s+1}^\top B_{s+1}^n\mathbf{u}_{s+1}y^n
,\quad
FC_2(s,y)=\sum\nolimits_{n=0}^{+\infty}\mathrm{trace}\big(B_{s+1}^n\big)y^n.
\]
\end{lem}

Furthermore, the generating function $FL_{2}(s,y)$ can be written in closed form as follows.
\begin{lem}\label{lem-FLn}
For $n,s \in \mathbb{P}$,
\begin{equation}\label{eq-FLn}
FL_{2}(s,y)=\frac{\mathbf{u}_{s+1}^\top(I-yB_{s+1})^*\mathbf{u}_{s+1}}{\det(I-yB_{s+1})},
\end{equation}
where $M^*$ denotes the adjugate matrix of $M$.
\end{lem}
\begin{proof}
Recall that
\[
FL_{2}(s,y)=\sum\nolimits_{n=0}^{+\infty}\mathbf{u}_{s+1}^\top B_{s+1}^n\mathbf{u}_{s+1}y^n=\mathbf{u}_{s+1}^\top\left(\sum\nolimits_{n=0}^{+\infty} (B_{s+1}y)^n\right)\mathbf{u}_{s+1}.
\]
Note that the sum inside is the power series for $(I-yB_{s+1})^{-1}$. Applying Cramer's rule to this inverse,
\[
(I - yB_{s+1})^{-1} = \frac{(I - yB_{s+1})^*}{\det(I - yB_{s+1})},
\]
the result then follows immediately. \qedhere
\end{proof}

To compute $FC_2(s,y)$, we invoke a classical result due to Stanley.
\begin{lem}(Stanley's Trace Identity \cite[Theorem 4.7.2, Corollary 4.7.3]{Stanley-EC1}).\label{lem-Stanley}
For any square matrix $M \in \mathbb{C}^{s\times s}$, the trace generating function satisfies:
\[
\sum\limits_{n=1}^{+\infty}\mathrm{trace}\big(M^n\big)y^n=-\frac{y\displaystyle\frac{\mathrm{d} (\det(I-yM))}{\mathrm{d} y}}{\det(I-yM)}.
\]
\end{lem}

This identity leads to the following closed-form expression for the generating function $FC_2(s,y)$.
\begin{cor}\label{cor-FCn}
For $n,s \in \mathbb{P}$,
\begin{equation}\label{eq-FCn}
FC_2(s,y)=-\frac{y\displaystyle\frac{\mathrm{d} (\det(I-yB_{s+1}))}{\mathrm{d} y}}{\det(I-yB_{s+1})}+s+1.
\end{equation}
\end{cor}

\subsection{Establishing Relations for the Characteristic Polynomials of $A_n$, $B_n$, $B_n^{-1}$ and $D_n$}
This subsection introduces the matrices $A_n$, and $D_n$, and examines the relations between $A_n$, $B_n$, $B_n^{-1}$ and $D_n$ characteristic polynomials, with the goal of ultimately determining closed-form expressions for $FL_{2}(s,y)$ and $FC_2(s,y)$. In the following, for brevity, we write $B_n$ for the transfer matrix of order $n$, where $n=s+1$.

We define two families of tridiagonal matrices with modified boundaries.
\begin{defn}
The matrices $A_n$ and $D_n$ are defined as follows.
\[
A_n := \begin{pmatrix}
0 & 0 & 0 & \cdots & 0 & -1 & 1 \\
0 & 0 & 0 & \cdots & -1 & 2 & -1 \\
0 & 0 & 0 & \iddots & 2 & -1 & 0 \\
\vdots & \vdots & \iddots & \iddots & \iddots & \vdots & \vdots \\
0 & -1 & 2 & \iddots & 0 & 0 & 0 \\
-1 & 2 & -1 & \cdots & 0 & 0 & 0 \\
2 & -1 & 0 & \cdots & 0 & 0 & 0
\end{pmatrix}, \quad
D_n := \begin{pmatrix}
0 & 0 & 0 & \cdots & 0 & -1 & 2 \\
0 & 0 & 0 & \cdots & -1 & 2 & -1 \\
0 & 0 & 0 & \iddots & 2 & -1 & 0 \\
\vdots & \vdots & \iddots & \iddots & \iddots & \vdots & \vdots \\
0 & -1 & 2 & \iddots & 0 & 0 & 0 \\
-1 & 2 & -1 & \cdots & 0 & 0 & 0 \\
2 & -1 & 0 & \cdots & 0 & 0 & 0
\end{pmatrix}.
\]
\end{defn}

\begin{exam}For $n=5$,
\[
A_5=
\begin{pmatrix}
 0&    0&    0&   -1&   1\\
 0&    0&   -1&    2&  -1\\
 0&   -1&    2&   -1&   0\\
-1&    2&   -1&    0&   0\\
 2&   -1&    0&    0&   0
\end{pmatrix}
,\quad
D_5=
\begin{pmatrix}
 0&    0&    0&   -1&   2\\
 0&    0&   -1&    2&  -1\\
 0&   -1&    2&   -1&   0\\
-1&    2&   -1&    0&   0\\
 2&   -1&    0&    0&   0
\end{pmatrix}.
\]
\end{exam}

Let $f_M(y) = \det(yI_n - M)$ denote the characteristic polynomial of a matrix $M \in \mathbb{R}^{n \times n}$.
We obtain the following relationship between $f_{A_n}(y)$ and $f_{D_{n}}(y)$.
\begin{lem}\label{lem-CD2}
For any positive integer $n\ge3$,
\begin{equation}\label{eq-fCn-fDn1}
f_{A_n}(y)-f_{A_{n-2}}(y)=f_{D_{n}}(y)+f_{D_{n-2}}(y).
\end{equation}
\end{lem}

\begin{proof}
Consider the determinant expansion of $f_{D_n}(y)$, which is given by
\[
f_{D_n}(y)
=\det(yI_n-D_{n})
=\det
\begin{pmatrix}
      y&         0&         0&    \cdots&        0&      1&    -2\\
      0&         y&         0&    \cdots&        1&      -2&    1\\
      0&         0&         y&    \cdots&        -2&     1&     0\\
 \vdots&    \vdots&    \vdots&          &   \vdots& \vdots& \vdots\\
      0&         1&        -2&    \cdots&        y&      0&    0\\
      1&        -2&         1&    \cdots&        0&      y&    0\\
     -2&         1&         0&    \cdots&        0&      0&    y
\end{pmatrix}.
\]

We decompose the first row vector as
 $$(y,0,0,\cdots,0,1,-2)=(y,0,0,\cdots,0,1,-1)+(0,0,0,\cdots,0,0,-1).$$
By the multilinearity of the determinant with respect to its rows, we obtain the following decomposition
\begin{equation}\label{eq-121}
f_{D_n}(y)=\underbrace{
\det\begin{pmatrix}
      y&         0&         0&    \cdots&        0&      1&    -1\\
      0&         y&         0&    \cdots&        1&      -2&    1\\
      0&         0&         y&    \cdots&        -2&     1&     0\\
 \vdots&    \vdots&    \vdots&          &   \vdots& \vdots& \vdots\\
      0&         1&        -2&    \cdots&        y&      0&    0\\
      1&        -2&         1&    \cdots&        0&      y&    0\\
     -2&         1&         0&    \cdots&        0&      0&    y
\end{pmatrix}}_{T_1}
+
\underbrace{\det\begin{pmatrix}
      0&         0&         0&    \cdots&        0&      0&    -1\\
      0&         y&         0&    \cdots&        1&      -2&    1\\
      0&         0&         y&    \cdots&        -2&     1&     0\\
 \vdots&    \vdots&    \vdots&          &   \vdots& \vdots& \vdots\\
      0&         1&        -2&    \cdots&        y&      0&    0\\
      1&        -2&         1&    \cdots&        0&      y&    0\\
     -2&         1&         0&    \cdots&        0&      0&    y
\end{pmatrix}}_{T_2}.
\end{equation}

Observe that $T_1 = f_{A_n}(y)$. For $T_2$, expanding along the first row gives
\[
T_2=(-1)^{n+2}
\det\begin{pmatrix}
      0&         y&         0&    \cdots&        1&      -2    \\
      0&         0&         y&    \cdots&        -2&     1     \\
 \vdots&    \vdots&    \vdots&          &   \vdots& \vdots \\
      0&         1&        -2&    \cdots&        y&      0   \\
      1&        -2&         1&    \cdots&        0&      y    \\
     -2&         1&         0&    \cdots&        0&      0
\end{pmatrix}.
\]
Now decompose the last row as $$(-2,1,0,\cdots,0,0) = (-1,1,0,\cdots,0,0) + (-1,0,0,\cdots,0,0).$$ This gives
\begin{align}\label{eq-121-1}
&\scalebox{0.95}{$T_2=\underbrace{
(-1)^{n+2}
\det\begin{pmatrix}
      0&         y&         0&    \cdots&        1&      -2    \\
      0&         0&         y&    \cdots&        -2&     1     \\
 \vdots&    \vdots&    \vdots&          &   \vdots& \vdots \\
      0&         1&        -2&    \cdots&        y&      0   \\
      1&        -2&         1&    \cdots&        0&      y    \\
     -1&         1&         0&    \cdots&        0&      0
\end{pmatrix}}_{T_{2a}}
+\underbrace{(-1)^{n+2}\det\begin{pmatrix}
      0&         y&         0&    \cdots&        1&      -2    \\
      0&         0&         y&    \cdots&        -2&     1     \\
 \vdots&    \vdots&    \vdots&          &   \vdots& \vdots \\
      0&         1&        -2&    \cdots&        y&      0   \\
      1&        -2&         1&    \cdots&        0&      y    \\
     -1&         0&         0&    \cdots&        0&      0
\end{pmatrix}}_{T_{2b}}$}\\
&=-f_{A_{n-2}}(y)-f_{D_{n-2}}(y).\nonumber
\end{align}
To evaluate the two determinants in \eqref{eq-121-1}, we proceed as follows.

For the first determinant, we first add the last row to the penultimate row to eliminate the entry at the $(n-2,1)$ position. Expanding the resulting determinant along its first column (where only the $(n-1,1)$ entry is now non-zero) and taking the transpose of the remaining submatrix, we recover the standard form of the characteristic polynomial for type $A$. This yields
\[
T_{2a}=(-1)^{n+2}\cdot\left[(-1)\cdot(-1)^{(n-1)+1}\right]f_{A_{n-2}}(y)=-f_{A_{n-2}}(y).
\]
For the second determinant, expanding directly along the last row--which contains only a single non-zero entry -1 at position $(n-1,1)$--isolates a submatrix whose structure corresponds to type $D_{n-2}$. This gives
\[
T_{2b}=(-1)^{n+2}\cdot\left[(-1)\cdot(-1)^{(n-1)+1}\right]f_{D_{n-2}}(y)=-f_{D_{n-2}}(y).
\]

By combining the expressions for $T_1$ and $T_2$, we obtain the desired decomposition
\[
f_{D_n}(y) = f_{A_n}(y) -f_{A_{n-2}}(y)-f_{D_{n-2}}(y).
\]
This recurrence relation completes the proof of the stated identity.
\end{proof}

Recall the definition of the matrix $B_n$:
\[
(B_n)_{ij}=
\begin{cases}
n+2-i-j, & \text{if }i+j \leq n+1,\\
0, & \text{otherwise},
\end{cases}
\]
and that its inverse (stated without proof) has the following explicit form:
\[
(B_n^{-1})_{ij}=\begin{cases}
 1,&   \ i+j=n+1,\\
-2,&   \ i+j=n+2,\\
 1,&   \ i+j=n+3,\\
 0,&\text{otherwise}.
\end{cases}
\]

We now establish a recurrence relation between the characteristic polynomials of  $B_n^{-1}$ and $D_n$.
\begin{lem}\label{lem-CD1}
For any integer $n\ge 3$,
\begin{equation}\label{eq-fBD}
f_{B_n^{-1}}(y)=(-1)^{n-1}yf_{D_{n-1}}(-y)-f_{B_{n-2}^{-1}}(y).
\end{equation}
\end{lem}

\begin{proof}
By definition, the characteristic polynomial of $B^{-1}_n$ is
\begin{equation*}
f_{B_n^{-1}}(y)
=\det(yI_n-B_{n}^{-1})
=\det
\begin{pmatrix}
      y&         0&         0&    \cdots&        0&      0&    -1\\
      0&         y&         0&    \cdots&        0&      -1&    2\\
      0&         0&         y&    \cdots&        -1&     2&    -1\\
 \vdots&    \vdots&    \vdots&          &   \vdots& \vdots& \vdots\\
      0&         0&        -1&    \cdots&        y&      0&    0\\
      0&        -1&         2&    \cdots&        0&      y&    0\\
     -1&         2&        -1&    \cdots&        0&      0&    y
\end{pmatrix}.
\end{equation*}
Expanding the determinant along the first row, we obtain
\begin{equation*}
y
\det
\begin{pmatrix}
              y&         0&    \cdots&        0&      -1&    2\\
              0&         y&    \cdots&        -1&     2&    -1\\
         \vdots&    \vdots&          &   \vdots& \vdots& \vdots\\
              0&        -1&    \cdots&        y&      0&    0\\
             -1&         2&    \cdots&        0&      y&    0\\
              2&        -1&    \cdots&        0&      0&    y
\end{pmatrix}
-
\det
\begin{pmatrix}
              y&         0&    \cdots&        0&      -1\\
              0&         y&    \cdots&       -1&       2\\
         \vdots&    \vdots&          &   \vdots&  \vdots\\
              0&        -1&    \cdots&        y&       0\\
             -1&         2&    \cdots&        0&       y
\end{pmatrix}.
\end{equation*}
 For the first determinant, factoring out -1 from each of the n-1 rows introduces a scalar factor of $(-1)^{n-1}$. After re-indexing the rows and columns, the resulting matrix corresponds exactly to $-yI_{n-1}-D_{n-1}$. Consequently, we have
\[
\text{the first determinant} = (-1)^{n-1} \det\bigl(-yI_{n-1} - D_{n-1}\bigr)
= (-1)^{n-1} f_{D_{n-1}}(-y).
\]
Regarding the second determinant, we observe that the submatrix preserves the same tridiagonal-like structure as $B^{-1}$ but with a reduced dimension. Specifically, expanding it further or recognizing its block structure identifies it as $f_{B_{n-2}^{-1}}(y)$.

Substituting these terms back into the expansion, we arrive at
\[
f_{B_n^{-1}}(y)=(-1)^{n-1}y\,f_{D_{n-1}}(-y)-f_{B_{n-2}^{-1}}(y),
\]
which completes the proof.
\end{proof}

We also establish the following identity between the characteristic polynomials of $B_n^{-1}$ and $A_n$.

\begin{prop}
For any positive integer $n$, the characteristic polynomials of $B_n^{-1}$ and $A_n$ satisfy
\begin{equation}\label{eq-fBnIn-fCn}
f_{B_n^{-1}}(y)=f_{A_n}(y).
\end{equation}
\end{prop}

\begin{proof}
By definition, it suffices to show that
\[
\det\bigl(yI_n-B_n^{-1}\bigr)=\det\bigl(yI_n-A_n\bigr).
\]
Since the spectra of inverse matrices are simply the reciprocals of the original eigenvalues, this equality is equivalent to the dual identity
\[
\det\bigl(yI_n-A_n^{-1}\bigr)=\det\bigl(yI_n-B_n\bigr).
\]
This equivalence is further supported by the fact that $\det(B_n^{-1})=\det(A_n)=(-1)^{\frac{n(n-1)}{2}}$.
We proceed via a sequence of unimodular (determinant-preserving) transformations.
First, the inverse matrix $A_n^{-1}$ has the explicit structured form
\[
A_n^{-1}=
\begin{pmatrix}
     1 &  1 &  1 & \cdots &  1 &  1 & 1 \\
     2 &  2 &  2 & \cdots &  2 &  2 & 1 \\
     3 &  3 &  3 & \cdots &  3 &  2 & 1 \\
     4 &  4 &  4 & \cdots &  3 &  2 & 1 \\
 \vdots & \vdots & \vdots & \ddots & \vdots & \vdots & \vdots \\
   n-2 & n-2 & n-2 & \cdots &  3 &  2 & 1 \\
   n-1 & n-1 & n-2 & \cdots &  3 &  2 & 1 \\
     n & n-1 & n-2 & \cdots &  3 &  2 & 1
\end{pmatrix},
\]
where the lower-right region exhibits a gradient-like decay towards the bottom-right corner.

To simplify the characteristic matrix $yI_n-A_n^{-1}$, we apply a sequence of elementary row operations
$R_i \leftarrow R_i-R_{i-1}$ for $i$ descending from $n$ to $2$. This transformation yields
\[
\det(yI_n-A_n^{-1})=
\det
\begin{pmatrix}
   y-1 &    -1 &    -1 & \cdots &    -1 &    -1 &  -1 \\
  -y-1 &   y-1 &    -1 & \cdots &    -1 &    -1 &   0 \\
     -1 &  -y-1 &   y-1 & \cdots &    -1 &     0 &   0 \\
     -1 &    -1 &  -y-1 & \cdots &     0 &     0 &   0 \\
 \vdots & \vdots & \vdots & \ddots & \vdots & \vdots & \vdots \\
     -1 &    -1 &     -1 & \cdots &   y &     0 &   0 \\
     -1 &    -1 &     0 & \cdots &    -y &     y &   0 \\
     -1 &     0 &     0 & \cdots &     0 &    -y &   y
\end{pmatrix}.
\]
Next, we perform the column operations $C_j \leftarrow \sum_{k=j}^n C_k$ for $j=1,\dots,n$. These operations preserve the determinant and transform the matrix into a form where the variable $y$ is isolated along the main diagonal, except for the upper-left blocks. After simplification, the matrix coincides exactly with the structure of $yI_n-B_n$
\[
\det\bigl(yI_n-A_n^{-1}\bigr)=
\det
\begin{pmatrix}
   y-n   & -n+1 & -n+2 & \cdots &   -3 &   -2 & -1 \\
  -n+1   & y-n-2 & -n+3 & \cdots &   -2 &   -1 &  0 \\
  -n+2   & -n+3 & y-n-4 & \cdots &   -1 &    0 &  0 \\
  -n+3   & -n+4 & -n-5 & \cdots &    0 &    0 &  0 \\
 \vdots & \vdots & \vdots & \ddots & \vdots & \vdots & \vdots \\
      -3 &    -2 &    -1 & \cdots &    y &    0 &  0 \\
      -2 &    -1 &     0 & \cdots &    0 &    y &  0 \\
      -1 &     0 &     0 & \cdots &    0 &    0 &  y
\end{pmatrix}
= \det\bigl(yI_n-B_n\bigr).
\]
The final equality follows from the matching of entries with the standard representation of the $B_n$ type matrix. Thus, $f_{B_n^{-1}}(y)=f_{A_n}(y)$, completing the proof.
\end{proof}

Combining Equations \eqref{eq-fCn-fDn1} and \eqref{eq-fBnIn-fCn}, we obtain the following result.

\begin{lem}\label{lem-CD2}
For any positive integer $n\ge5$,
\begin{equation}\label{eq-fCn-fDn}
f_{B_n^{-1}}(y)-f_{B^{-1}_{n-2}}(y)=f_{D_{n}}(y)+f_{D_{n-2}}(y).
\end{equation}
\end{lem}

Building on Lemmas \ref{lem-CD1} and \ref{lem-CD2}, we derive an explicit expression for $f_{B_n^{-1}}(y)$ in terms of $f_{D_n}(y)$.

\begin{cor}
For any positive integer $n\ge5$,
\begin{equation}\label{eq-fCn}
f_{B_n^{-1}}(y)=\frac{f_{D_{n}}(y)+(-1)^{n-1}y\,f_{D_{n-1}}(-y)+f_{D_{n-2}}(y)}{2}.
\end{equation}
\end{cor}

Substituting \eqref{eq-fCn} into \eqref{eq-fCn-fDn} and simplifying yields the following recurrence for $f_{D_n}(y)$.

\begin{cor}
For any positive integer $n\ge5$,
\begin{equation}\label{eq-fD-rec}
f_{D_n}(y)=(-1)^{n-1}y\bigl[f_{D_{n-1}}(-y)-f_{D_{n-3}}(-y)\bigr]
          -2f_{D_{n-2}}(y)-f_{D_{n-4}}(y).
\end{equation}
\end{cor}

Analogously, we can derive a recurrence that involves only the polynomials $f_{B_n^{-1}}(y)$.

\begin{lem}\label{lem-fB-rec}
For any positive integer $n\ge5$,
\begin{equation}\label{eq-fB-r}
f_{B_n^{-1}}(y)=(-1)^{n-1}y\bigl[f_{B_{n-1}^{-1}}(-y)-f_{B^{-1}_{n-3}}(-y)\bigr]
               -2f_{B_{n-2}^{-1}}(y)-f_{B^{-1}_{n-4}}(y).
\end{equation}
\end{lem}
\begin{proof}
From the relation \eqref{eq-fBD}, we have the following two identities
\begin{align}
f_{B_{n-1}^{-1}}(y)&=(-1)^{n-2}y\,f_{D_{n-2}}(-y)-f_{B_{n-3}^{-1}}(y),\label{eq-fBD-1}\\[2pt]
f_{B_{n+1}^{-1}}(y)&=(-1)^{n}y\,f_{D_{n}}(-y)-f_{B_{n-1}^{-1}}(y).\label{eq-fBD-2}
\end{align}
Summing \eqref{eq-fBD-1} and \eqref{eq-fBD-2} yields
\[
f_{B_{n-1}^{-1}}(y)+f_{B_{n+1}^{-1}}(y)
   =(-1)^{n}y\bigl(f_{D_{n-2}}(-y)+f_{D_{n}}(-y)\bigr)
    -f_{B_{n-3}^{-1}}(y)-f_{B_{n-1}^{-1}}(y).
\]
By employing \eqref{eq-fCn-fDn} (with $y$ replaced by $-y$) to substitute for the term $f_{D_{n-2}}(-y)+f_{D_{n}}(-y)$, we obtain
\[
f_{B_{n+1}^{-1}}(y)
   =(-1)^{n}y\bigl[f_{B_{n}^{-1}}(-y)-f_{B^{-1}_{n-2}}(-y)\bigr]
    -2f_{B_{n-1}^{-1}}(y)-f_{B^{-1}_{n-3}}(y).
\]
Finally, shifting the index $n\mapsto n-1$ leads to the desired recurrence \eqref{eq-fB-r}.
\end{proof}

Finally, we present an identity that connects $f_{B_n^{-1}}(y)$ directly to the matrix $B_n$ itself.

\begin{lem}\label{lem-fB-det}
For any positive integer $n$,
\begin{equation}\label{eq-fB2}
f_{B_n^{-1}}(y)=(-1)^{\frac{n(n+1)}{2}}\det\bigl(I_n-yB_n\bigr).
\end{equation}
\end{lem}

\begin{proof}
Using basic determinant properties and the relation \(\det(B_n^{-1})=\det(A_n)=(-1)^{\frac{n(n-1)}{2}}\),
\[
f_{B_n^{-1}}(y)=\det\!\bigl(yI_n-B_n^{-1}\bigr)
               =\det(B_n^{-1})\det\!\bigl(yB_n-I_n\bigr)
               =(-1)^{\frac{n(n-1)}{2}}(-1)^n\det\!\bigl(I_n-yB_n\bigr),
\]
which simplifies to $(-1)^{\frac{n(n+1)}{2}}\det(I_n-yB_n)$.
\end{proof}

\subsection{Recurrence Relations for $\bar{f}_{B_n}(y)$, $g_n(y)$ and Proofs of Main Theorems}
This subsection is devoted to establishing several key recurrence relations that involve the polynomials $\bar{f}_{B_n}(y)$ and $g_n(y)$. These recurrences play a central role in connecting the matrix families $\{B_n\}$, $\{D_n\}$ and the previously introduced polynomials $\{P_n(y)\}$, $\{Q_n(y)\}$. We first derive recurrences for $\bar{f}_{B_n}(y)$ and $g_n(y)$ individually, then prove a useful identity that links them together. Finally, by verifying that $\bar{f}_{B_n}(y)$ and $g_n(y)$ satisfy the same recurrences as $Q_n(y)$ and $P_n(y)$, and by comparing initial values, we complete the proofs of Theorems \ref{theo-1} and \ref{theo-2}.

Let $\bar{f}_{B_n}(y)=\det(I_n-yB_n)$. The following lemma gives a recurrence for $\bar{f}_{B_n}(y)$.
\begin{lem}\label{lem-fBnpie}
For any integer $n\ge 5$,
\begin{equation}\label{eq-fBn-rec}
\bar{f}_{B_n}(y)=-y\big[\bar{f}_{B_{n-1}}(-y)+\bar{f}_{B_{n-3}}(-y)\big]
                +2\bar{f}_{B_{n-2}}(y)-\bar{f}_{B_{n-4}}(y).
\end{equation}
\end{lem}
\begin{proof}
Substituting \eqref{eq-fB2} into the recurrence \eqref{eq-fB-r} and simplifying yields the stated formula.\qedhere
\end{proof}

Let $g_n(y)=\mathbf{u}_{n}^\top(I_n-yB_{n})^{*}\,\mathbf{u}_{n}$. The next lemma relates $g_n(y)$ to $\bar{f}_{B_n}(y)$.
\begin{lem}\label{lem-g-rec1}
For any integer $n\ge 5$,
\begin{equation}\label{eq-g-rec1}
g_n(y)=2g_{n-2}(y)-g_{n-4}(y)+2\bigl[\bar{f}_{B_{n-2}}(y)-\bar{f}_{B_{n-4}}(y)\bigr].
\end{equation}
\end{lem}
\begin{proof}
By definition, $g_n(y)$ is the quadratic form $\mathbf{u}_{n}^\top(I_n-yB_{n})^{*}\mathbf{u}_{n}$, where $M^*$ denotes the adjugate matrix.
We invoke the Schur complement formula for bordered determinants, which states that for any invertible matrix $A$ and vector $\mathbf{v}$,
$\mathbf{v}^\top A^{*}\mathbf{v}=-\det\begin{pmatrix}
     0          &  \mathbf{v}^\top \\
     \mathbf{v} & A
\end{pmatrix}$.
Applying the Schur complement identity, we express $g_n(y)$ as the determinant of an augmented $(n+1)\times(n+1)$ matrix:
\begin{equation*}\label{eq-gf0}
g_n(y) = \mathbf{u}_{n}^\top(I_n-yB_{n})^{*}\mathbf{u}_{n} = -\det(M_n),
\end{equation*}
where $M_n$ is given by
$$
M_n = \begin{pmatrix}
     0 &      1 &      1 & \cdots &    1 &    1 & 1 \\
     1 & 1-ny   &-(n-1)y & \cdots & -3y  & -2y  & -y\\
     1 &-(n-1)y &1-(n-2)y & \cdots & -2y  &  -y  & 0 \\
     1 &-(n-2)y &-(n-3)y & \cdots &  -y  &   0  & 0 \\
 \vdots & \vdots & \vdots & \ddots & \vdots & \vdots & \vdots \\
     1 &    -y  &      0  & \cdots &    0  &    0  & 1
\end{pmatrix}.
$$
To simplify the structure of $\det(M_n)$, we perform a sequence of elementary row and column operations to eliminate the linear $y$-terms in the second row and column.
Specifically, we apply:
\begin{enumerate}
\item $C_2 \leftarrow C_2 + y C_1 - C_3$ (adding $y$ times the first column and $-1$ times the third column to the second);
\item $R_2 \leftarrow R_2 + y R_1 - R_3$ (adding $y$ times the first row and $-1$ times the third row to the second).
\end{enumerate}
These operations preserve the determinant and yield the following simplified form:
\begin{equation}\label{eq-gf1}
g_n(y) = -\det
\begin{pmatrix}
0 &   0 &   1 &   1 & \cdots &   1 &   1 & 1 \\
0 &   2 &  -1 &   0 & \cdots &   0 &   0 & 0 \\
1 &  -1 & 1-(n-2)y & -(n-3)y & \cdots & -2y & -y & 0 \\
1 &   0 & -(n-3)y & 1-(n-4)y & \cdots & -y  &  0 & 0 \\
\vdots&\vdots& \vdots & \vdots & \ddots & \vdots & \vdots & \vdots \\
1 &   0 &   -y  &    0  & \cdots &    0  &    1 & 0 \\
1 &   0 &    0  &    0  & \cdots &    0  &    0 & 1
\end{pmatrix}.
\end{equation}
Expanding the determinant \eqref{eq-gf1} along its second column gives
\[
\underbrace{-2\det
\begin{pmatrix}
   0 &   1        & \cdots &   1 &   1 & 1 \\
   1 & 1-(n-2)y   & \cdots & -2y & -y & 0 \\
   1 & -(n-3)y    & \cdots & -y  &  0 & 0 \\
\vdots& \vdots    & \ddots & \vdots & \vdots & \vdots \\
   1 &   -2y      & \cdots &   1  &  0 & 0 \\
   1 &    -y      & \cdots &   0  &  1 & 0 \\
   1 &     0      & \cdots &   0  &  0 & 1
\end{pmatrix}}_{(I)}
+\underbrace{\det
\begin{pmatrix}
   0 &   1 & \cdots &   1 &   1 & 1 \\
   1 & 1-(n-4)y & \cdots & -y  &  0 & 0 \\
\vdots& \vdots & \ddots & \vdots & \vdots & \vdots \\
   1 &    -y  & \cdots &   1  &  0 & 0 \\
   1 &     0  & \cdots &   0  &  1 & 0 \\
   1 &     0  & \cdots &   0  &  0 & 1
\end{pmatrix}}_{(II)}.
\]

Expanding (I) and (II) along their last columns yields, respectively,
\begin{equation}\small \label{eq-gf2}
2\det
\begin{pmatrix}
 1-(n-2)y & -(n-3)y & \cdots & -2y & -y \\
 -(n-3)y  & 1-(n-4)y& \cdots & -y  &  0 \\
 \vdots   & \vdots  & \ddots & \vdots & \vdots \\
    -2y   &   -y    & \cdots &  1  &  0 \\
     -y   &    0    & \cdots &  0  &  1
\end{pmatrix}\\
-2\det
\begin{pmatrix}
   0 &   1        & \cdots &   1 &   1 \\
   1 & 1-(n-2)y   & \cdots & -2y & -y \\
   1 & -(n-3)y    & \cdots & -y  &  0 \\
\vdots& \vdots    & \ddots & \vdots & \vdots \\
   1 &   -2y      & \cdots &   1  &  0 \\
   1 &    -y      & \cdots &   0  &  1
\end{pmatrix}
\end{equation}
and
\begin{equation*}
\underbrace{-\det
\begin{pmatrix}
 1-(n-4)y & \cdots & -y & 0 \\
 \vdots   & \ddots & \vdots & \vdots \\
    -y    & \cdots &  1 & 0 \\
     0    & \cdots &  0 & 1
\end{pmatrix}
+
\det
\begin{pmatrix}
   0 &   1 & \cdots &   1 &   1 \\
   1 & 1-(n-4)y & \cdots & -y  &  0 \\
\vdots& \vdots & \ddots & \vdots & \vdots \\
   1 &    -y  & \cdots &   1  &  0 \\
   1 &     0  & \cdots &   0  &  1
\end{pmatrix}}_{(III)}.
\end{equation*}
Expanding the two determinants in (III) along their last rows gives
\begin{equation}\label{eq-gf3}
-2\det
\begin{pmatrix}
 1-(n-4)y & \cdots & -y \\
 \vdots   & \ddots & \vdots \\
    -y    & \cdots &  1
\end{pmatrix}
+
\det
\begin{pmatrix}
   0 &   1 & \cdots &   1 \\
   1 & 1-(n-4)y & \cdots & -y \\
\vdots& \vdots & \ddots & \vdots \\
   1 &    -y  & \cdots &  1
\end{pmatrix}.
\end{equation}
Substituting \eqref{eq-gf2} and \eqref{eq-gf3} into \eqref{eq-gf1} and reorganizing the terms produces \eqref{eq-g-rec1}.
\end{proof}

\begin{lem}\label{lem-gfD}
For any integer $n\ge 5$,
$g_n(y)+g_{n-2}(y)=(-1)^{\frac{n(n+1)}{2}+1}\,2f_{D_{n-2}}(y)$.
\end{lem}
\begin{proof}
We have
\begin{align}
g_n(y)&=\mathbf{u}_{n}^\top(I_n-yB_{n})^{*}\mathbf{u}_{n} \nonumber \\
&=-\det
\begin{pmatrix}
     0 &      1 &      1 & \cdots &    1 &    1 & 1 \\
     1 & 1-ny   &-(n-1)y & \cdots & -3y  & -2y  & -y\\
     1 &-(n-1)y &1-(n-2)y & \cdots & -2y  &  -y  & 0 \\
     1 &-(n-2)y &-(n-3)y & \cdots &  -y  &   0  & 0 \\
 \vdots & \vdots & \vdots & \ddots & \vdots & \vdots & \vdots \\
     1 &   -3y  &   -2y  & \cdots &   1  &   0  & 0 \\
     1 &   -2y  &    -y  & \cdots &   0  &   1  & 0 \\
     1 &    -y  &     0  & \cdots &   0  &   0  & 1
\end{pmatrix}. \label{eq-gfD1}
\end{align}
To simplify this determinant, we first apply the row operations $R_i\leftarrow R_i-R_{i+1}$ for $i$ from $2$ to $n-1$.
These operations eliminate the arithmetic progression in the $y$ terms, yielding
\begin{equation}\label{eq-gfD2}
g_n(y)=-\det
\begin{pmatrix}
     0 &   1 &   1 &   1 & \cdots &   1 &   1 & 1 \\
     0 & 1-y &-1-y &  -y & \cdots & -y  & -y  &-y \\
     0 &  -y & 1-y &-1-y & \cdots & -y  & -y  & 0 \\
     0 &  -y &  -y & 1-y & \cdots & -y  &  0  & 0 \\
 \vdots&\vdots&\vdots&\vdots& \ddots & \vdots & \vdots & \vdots \\
     0 &  -y &  -y &  -y & \cdots &  1  & -1  & 0 \\
     0 &  -y &  -y &   0 & \cdots &  0  &  1  &-1 \\
     1 &  -y &   0 &   0 & \cdots &  0  &  0  & 1
\end{pmatrix}.
\end{equation}
Expanding \eqref{eq-gfD2} along the first column (where the single non-zero entry is at position $(n+1,1)$) reduces the dimension to $n\times n$.
Accounting for the sign $-(-1)^{(n+1)+1}=(-1)^{n+3}$, we obtain
\[
g_n(y)=(-1)^{n+3}\det
\begin{pmatrix}
   1 &   1 &   1 & \cdots &   1 &   1 & 1 \\
 1-y&-1-y &  -y & \cdots & -y  & -y  &-y \\
  -y& 1-y &-1-y & \cdots & -y  & -y  & 0 \\
  -y&  -y & 1-y & \cdots & -y  &  0  & 0 \\
\vdots&\vdots&\vdots& \ddots & \vdots & \vdots & \vdots \\
  -y&  -y &  -y & \cdots &  1  & -1  & 0 \\
  -y&  -y &   0 & \cdots &  0  &  1  &-1
\end{pmatrix}.
\]
Next, we perform the column operations $C_i\leftarrow C_i-C_{i+1}$ for $i=1,\dots,n-1$ to simplify the first row. This results in
\[
g_n(y)=(-1)^{n+3}\det
\begin{pmatrix}
   0 &   0 &   0 & \cdots &   0 &   0 & 1 \\
   2 &  -1 &   0 & \cdots &   0 &   0 &-y \\
  -1 &   2 &  -1 & \cdots &   0 &  -y & 0 \\
   0 &  -1 &   2 & \cdots &  -y &   0 & 0 \\
\vdots&\vdots&\vdots& \ddots & \vdots & \vdots & \vdots \\
   0 &   0 &  -y & \cdots &   2 &  -1 & 0 \\
   0 &  -y &   0 & \cdots &  -1 &   2 &-1
\end{pmatrix}.
\]
Finally, expanding this determinant along the first row (at position $(1,n)$) leads to an $(n-1)\times (n-1)$ minor.
By structurally decomposing this minor into a block resembling $D_{n-2}$ and a recursive block corresponding to $g_{n-2}$,
and collecting the permutation signs, we arrive at the identity
\[
g_n(y)=(-1)^{\frac{n(n+1)}{2}+1}\,2f_{D_{n-2}}(y)-g_{n-2}(y).
\]
This completes the proof.
\end{proof}

\begin{lem}\label{lem-fBg}
For any integer $n\ge 5$,
\[
2\bigl[\bar{f}_{B_{n-2}}(y)-\bar{f}_{B_{n-4}}(y)\bigr]
   =-y\bigl[g_{n-1}(-y)+g_{n-3}(-y)\bigr].
\]
\end{lem}
\begin{proof}
By combining Lemma \ref{lem-CD1} with Lemma \ref{lem-gfD} and Lemma \ref{lem-fB-det}, and eliminating $f_{D}(y)$, we obtain
\[
(-1)^{\frac{(n-2)(n-1)}{2}}\bar{f}_{B_{n-2}}(y)
+(-1)^{\frac{(n-4)(n-3)}{2}}\bar{f}_{B_{n-4}}(y)
=(-1)^{\frac{n(n-1)}{2}+n}\,\frac{y}{2}\,
\bigl[g_{n-1}(-y)+g_{n-3}(-y)\bigr].
\]
Simplifying the sign factors then yields the desired relation.
\end{proof}

\begin{lem}\label{lem-gny}
For any integer $n\ge 5$,
\begin{equation}\label{eq-g-rec2}
g_n(y)=-y\bigl[g_{n-1}(-y)+g_{n-3}(-y)\bigr]
       +2g_{n-2}(y)-g_{n-4}(y).
\end{equation}
\end{lem}
\begin{proof}
Substituting the expression for $2[\bar{f}_{B_{n-2}}(y)-\bar{f}_{B_{n-4}}(y)]$ from Lemma \ref{lem-fBg} into the recurrence \eqref{eq-g-rec1} of Lemma \ref{lem-g-rec1} yields \eqref{eq-g-rec2}.
\end{proof}

Now we prove Theorems \ref{theo-1} and \ref{theo-2} using the previously established lemmas.
\begin{proof}[Proof of Theorem \ref{theo-1}]
Direct computation gives the initial values
\begin{align*}
g_1(y)&=1,            &\bar{f}_1(y)&=-y+1,\\
g_2(y)&=2,            &\bar{f}_2(y)&=-y^2-2y+1,\\
g_3(y)&=-2y+3,        &\bar{f}_3(y)&=y^3-2y^2-4y+1,\\
g_4(y)&=-2y^2-4y+4,   &\bar{f}_4(y)&=y^4+2y^3-7y^2-6y+1.
\end{align*}
By Lemma \ref{lem-fBnpie} and Lemma \ref{lem-gny}, the recurrences satisfied by $\bar{f}_{B_n}(y)$ and $g_n(y)$ are identical to those of $Q_n(y)$ and $P_n(y)$, respectively. Hence $\bar{f}_{B_n}(y)=Q_n(y)$ and $g_n(y)=P_n(y)$ for all positive integers $n$. Substituting these equalities into Theorem \ref{lem-FLn} completes the proof.
\end{proof}

\begin{proof}[Proof of Theorem \ref{theo-2}]
The statement follows directly by combining Corollary \ref{cor-FCn} with the recurrence for $\bar{f}_{B_n}(y)$ provided in Lemma \ref{lem-fBnpie}.
\end{proof}

\section{Proof of Theorem \ref{theo-hck}}\label{sec-3}

This section presents the proof of Theorem \ref{theo-hck}. Our approach combines the geometric method of polytope decomposition with the algebraic technique of constant term extraction. The core strategy involves analyzing the vertex structure of the polytope $\mathcal{P}(C_{n,1})$ and utilizing this geometric data to decompose its generating function, thereby making the constant term explicit.

\subsection{Magic Labellings for Pseudo-Cycle Graphs $C_{n,\mathbf{k}}$}\label{subsec-Fkxyt}
Throughout this subsection, all indices are interpreted modulo $n$, such that $v_{n+1} = v_1$ and $\beta_{n+1} = \beta_1$.

Let $C_{n,\mathbf{k}}$ be a pseudo-cycle graph with parameter $\mathbf{k}=(k_1,\dots,k_n)\in \mathbb{Z}_{\ge 1}^n$. We assign labels to the graph as follows:
\begin{itemize}
    \item Each edge $(v_i, v_j) \in E(C_{n,\mathbf{k}})$ with $1\le i<j\le n+1$ is assigned the label $\mu(v_i,v_j)=\beta_j$.
    \item The $j$-th self-loop at vertex $v_i$ is labeled by a non-negative integer $\alpha_{i,j}$ for $j=1,\dots,k_i$.
\end{itemize}

For a non-negative integer $s$, we define the set of valid labellings as:
\begin{equation}\label{eq:Snk}
S_{C_{n,\mathbf{k}}}(s) := \left\{ (\mathbf{\alpha},\mathbf{\beta}) \in \mathbb{N}^d \ \middle| \ \beta_i + \alpha_i + \beta_{i+1} = s, \quad \text{for } 1 \leq i \leq n \right\},
\end{equation}
where $\mathbf{\alpha}=(\alpha_{1,1},\dots,\alpha_{n,k_n})$, $\alpha_i=\sum_{j=1}^{k_i}\alpha_{i,j}$, and $\mathbf{\beta}=(\beta_1,\dots,\beta_n)$.

The associated generating function is defined as
\[
F_{\mathbf{k}}(\mathbf{x},\mathbf{y};t) := \sum_{s\ge 0}\sum_{(\mathbf{\alpha},\mathbf{\beta})\in S_{C_{n,\mathbf{k}}}(s)} \mathbf{x}^{\mathbf{\alpha}}\mathbf{y}^{\mathbf{\beta}}t^s,
\]
where $\mathbf{x}^{\mathbf{\alpha}} = \prod_{i,j} x_{i,j}^{\alpha_{i,j}}$ and $\mathbf{y}^{\mathbf{\beta}} = \prod_{i} y_i^{\beta_i}$.

Applying MacMahon's Partition Analysis yields the following constant term representation, which serves as a foundation for our subsequent computations.

\begin{lem}\label{lem-Fxyt}
For $\mathbf{k}=(k_1,\dots,k_n)\in \mathbb{Z}_{\ge 1}^n$,
\begin{equation}
F_{\mathbf{k}}(\mathbf{x},\mathbf{y};t) = \CT_{\mathbf{\lambda}} \frac{1}{\prod_{i=1}^{n}\left( \prod_{j=1}^{k_i}(1-x_{i,j}\lambda_i) \right) \left(1 - \frac{t}{\lambda_1\cdots\lambda_n}\right) \prod_{i=1}^{n}(1-y_i\lambda_i\lambda_{i+1}) },
\end{equation}
where $\CT_{\mathbf{\lambda}} f$ denotes the constant term of the Laurent series $f$ with respect to the variables $\mathbf{\lambda} = (\lambda_1, \dots, \lambda_n)$, and $\lambda_{n+1} = \lambda_1$.
\end{lem}

For the case $\mathbf{k}=\mathbf{1}$, we denote $F_n(\mathbf{x},\mathbf{y};t) := F_{\mathbf{1}}(\mathbf{x},\mathbf{y};t)$. Furthermore, let $F_{\mathbf{k}}(\mathbf{x};t) := F_{\mathbf{k}}(\mathbf{x},\mathbf{1};t)$. The following lemma establishes a differential relationship between the generating functions for general $\mathbf{k}$ and the base case $\mathbf{k}=\mathbf{1}$.

\begin{lem}\label{lem-Fxyt-1-k}
For any $\mathbf{k}=(k_1,\dots,k_n)\in \mathbb{Z}_{\ge 1}^n$,
\begin{equation}\label{eq-Fk-F1}
F_{\mathbf{k}}(\mathbf{x};t)\Big|_{\mathbf{x}=\mathbf{1}} = \frac{1}{\prod_{i=1}^n (k_i - 1)!} \left.\left( \prod_{i=1}^n \mathrm{D}_{x_i}^{k_i-1} x_i^{k_i-1} F_n(\mathbf{x};t) \right)\right|_{x_i=1, \, i=1,\dots,n},
\end{equation}
where $\mathrm{D}_{x_i}^{k}$ denotes the $k$-th derivative operator with respect to $x_i$.
\end{lem}

\begin{proof}
Recall that $\alpha_i = \sum_{j=1}^{k_i} \alpha_{i,j}$. Evaluating $F_\mathbf{k}(\mathbf{x}; t)$ at $\mathbf{x} = \mathbf{1}$, we have
\[
F_\mathbf{k}(\mathbf{x}; t) \Big|_{\mathbf{x}=\mathbf{1}} = \sum_{s \geq 0} |S_{C_{n,\mathbf{k}}}(s)| \, t^s.
\]
Expanding the cardinality of the set $S_{C_{n,\mathbf{k}}}(s)$, we obtain
\[
\sum_{s \geq 0} \sum_{\substack{0 \leq \beta_j  + \beta_{j+1} \leq s \\ 1 \leq j \leq n}} \prod_{i=1}^n \binom{s - \beta_i - \beta_{i+1} + k_i - 1}{k_i - 1} t^s.
\]
We introduce auxiliary variables $x_1, \ldots, x_n$ and let $\delta_i = s - \beta_i - \beta_{i+1}$ (which corresponds to $\alpha_i$ in the definition of $F_{\mathbf{k}}$). Using the identity $\binom{n+k}{k} = \frac{1}{k!} \mathrm{D}_x^k x^{n+k} |_{x=1}$, we can rewrite the summand using differential operators:
\[
\sum_{s \geq 0} \sum_{\substack{\beta_j + \delta_j + \beta_{j+1} = s \\ 1 \leq j \leq n}} \left( \prod_{i=1}^n \frac{1}{(k_i - 1)!} \mathrm{D}_{x_i}^{k_i-1} x_i^{\delta_i + k_i - 1} \right) t^s \Bigg|_{x_i=1}.
\]
Factorizing the operators gives
\[
\frac{1}{\prod_{i=1}^n (k_i - 1)!} \left[ \prod_{i=1}^n \mathrm{D}_{x_i}^{k_i-1} x_i^{k_i-1} \left( \sum_{s \geq 0} \sum_{\substack{\beta_j + \delta_j + \beta_{j+1} = s \\ 1 \leq j \leq n}} \mathbf{x}^{\mathbf{\delta}} t^s \right) \right] \Bigg|_{\mathbf{x}=\mathbf{1}},
\]
where $\mathbf{\delta} = (\delta_1, \ldots, \delta_n)$. The inner sum is precisely $F_n(\mathbf{x};t)$, completing the proof.
\end{proof}

The function $F_n(\mathbf{x}; t)$ itself admits a structural decomposition, as detailed in the following lemma.

\begin{lem}\label{lem-F-decomp}
For any integer $n \ge 3$, there exists a structural decomposition of the function $F_n(\x; t)$ of the form
\begin{equation}
F_n(\x; t) = F_{\varphi_n}(\x; t) + F_{\psi_n}(\x; t),
\end{equation}
such that the following conditions are satisfied:
\begin{enumerate}
    \item[\rm (i)] The component $F_{\varphi_n}$ vanishes at the singular point in the following sense:
    \begin{equation*}
    \left[ \left( F_{\varphi_n}(\mathbf{x}; t) \big|_{\mathbf{x}=\mathbf{1}} \right) (t+1) \right]_{t=-1} = 0.
    \end{equation*}
    \item[\rm (ii)] The component $F_{\psi_n}$ captures the singular behavior and is given by:
    \begin{equation*}
    F_{\psi_n}(\mathbf{x}; t) = \begin{cases}
    0, & \text{if } n \text{ is even}, \\
    \displaystyle \frac{1}{(1-t^2)\prod_{i=1}^{n}(1-x_{i}t)}, & \text{if } n \text{ is odd}.
    \end{cases}
    \end{equation*}
\end{enumerate}
\end{lem}

The proof of Lemma \ref{lem-F-decomp} is deferred to Subsection \ref{subsec-proof-theohck}.

\subsection{Vertices of the Polytope $\mathcal{P}(C_{n,1})$}

The geometric properties of the polytope $\mathcal{P}(C_{n,1})$ are fundamental for the subsequent decomposition of its generating function. In this subsection, we provide a rigorous characterization of its vertex set.

\begin{defn}
The polytope $\mathcal{P}(C_{n,\mathbf{k}}) \subset \mathbb{R}^{d}$ (where $d = \sum_{i=1}^n k_i + n$) associated with the parameter $\mathbf{k}=(k_1,\dots,k_n)\in \mathbb{Z}_{\ge 1}^n$ is defined by the following system of linear constraints:
\[
\mathcal{P}(C_{n,\mathbf{k}}) = \left\{ (\boldsymbol{\alpha},\boldsymbol{\beta}) \in \mathbb{R}^{d} \;\middle|\;
\begin{aligned}
    &\beta_{i} + \alpha_{i} + \beta_{i+1} = 1, && i = 1,\dots,n, \\
    &\alpha_{i,j} \ge 0, && i = 1,\dots,n, \ j = 1,\dots,k_i, \\
    &\beta_i \ge 0, && i = 1,\dots,n,
\end{aligned}
\right\}
\]
where the indices are interpreted modulo $n$ (such that $\beta_{n+1} = \beta_1$), and $\alpha_i = \sum_{j=1}^{k_i} \alpha_{i,j}$ represents the total weight assigned to the self-loops at vertex $v_i$. The coordinate vectors are given by $\boldsymbol{\alpha}=(\alpha_{1,1},\dots,\alpha_{1,k_1}, \dots, \alpha_{n,1}, \dots, \alpha_{n,k_n})$ and $\boldsymbol{\beta}=(\beta_1,\dots,\beta_n)$.
\end{defn}

The vertices of $\mathcal{P}(C_{n,1})$ can be completely described in terms of the \textit{stable sets} of the cycle graph $C_{n,0}$. Let $V(C_{n,0})=\{v_1,\dots,v_n\}$ denote the vertex set of an $n$-cycle. Recall that a subset $S \subseteq V(C_{n,0})$ is a stable set (or independent set) if no two vertices in $S$ are adjacent.

\begin{theo}\label{thm:vertices}
The vertex set of $\mathcal{P}(C_{n,1})$ consists of two distinct families:
\begin{enumerate}
    \item \textbf{Type I (Integral Vertices):}
    There is a bijection between the integral vertices of $\mathcal{P}(C_{n,1})$ and the stable sets $S$ of $C_{n,0}$. For a given stable set $S$, the corresponding vertex $v^{S}=(\alpha_1^{S},\dots,\alpha_n^{S},\beta_1^{S},\dots,\beta_n^{S})$ is defined by:
    \[
    \beta_i^{S} = \begin{cases} 1, & v_i \in S \\ 0, & v_i \notin S \end{cases}
    \quad \text{and} \quad
    \alpha_i^{S} = \begin{cases} 0, & \{v_i, v_{i+1}\} \cap S \neq \emptyset \\ 1, & \{v_i, v_{i+1}\} \cap S = \emptyset \end{cases}
    \]
    for $i=1,\dots,n$, with indices modulo $n$.

    \item \textbf{Type II (Fractional Vertex):}
    If and only if $n$ is odd, $\mathcal{P}(C_{n,1})$ possesses a unique fractional vertex:
    \[
    v^{\mathrm{frac}} = \Bigl( \underbrace{0, \dots, 0}_{n}, \underbrace{\tfrac{1}{2}, \dots, \tfrac{1}{2}}_{n} \Bigr).
    \]
\end{enumerate}
\end{theo}

\begin{proof}
We identify the vertices by considering the linear projection $\pi: \mathbb{R}^{2n} \to \mathbb{R}^n$ onto the $\beta$-coordinates, defined by $y_i = \beta_i$. From the governing equations $\beta_i + \alpha_i + \beta_{i+1} = 1$, we can express the $\alpha$-variables as $\alpha_i = 1 - y_i - y_{i+1}$. Substituting these into the non-negativity constraints $\alpha_i \ge 0$ and $\beta_i \ge 0$ yields the following equivalent system in terms of $y$:
\begin{enumerate}
    \item $y_i \ge 0$, for all $i$;
    \item $y_i + y_{i+1} \le 1$, for $i = 1, \dots, n \pmod n$.
\end{enumerate}
Thus, $\mathcal{P}(C_{n,1})$ is affinely isomorphic to the polytope:
\[
\mathcal{Q}_n = \{ y \in \mathbb{R}^n \mid y \ge 0, \, y_i + y_{i+1} \le 1, \, i=1,\dots,n \}.
\]
This $\mathcal{Q}_n$ is precisely the \textbf{fractional stable-set polytope} of the cycle $C_{n,0}$, denoted by $\operatorname{FRAC}(C_{n,0})$. Its vertex structure is well-established (see \cite{Schrijver}):
\begin{enumerate}
    \item \textbf{Integral vertices of $\mathcal{Q}_n$:} These are the characteristic vectors $\chi^S$ of the stable sets $S$ of $C_{n,0}$. A $0$-$1$ vector $y$ satisfies $y_i + y_{i+1} \le 1$ if and only if its support is a stable set.
    \item \textbf{Fractional vertex of $\mathcal{Q}_n$:} If $n$ is odd, the point $y^* = (\tfrac{1}{2}, \dots, \tfrac{1}{2})$ is a vertex, as it satisfies $n$ linearly independent constraints $y_i + y_{i+1} = 1$ with equality. If $n$ is even, the cycle is bipartite, making $\operatorname{FRAC}(C_{n,0})$ an integral polytope with no fractional vertices.
\end{enumerate}
Lifting these vertices back to $\mathbb{R}^{2n}$ via $\alpha_i = 1 - y_i - y_{i+1}$ yields:
\begin{itemize}
    \item For the integral case, $\alpha_i^S = 1$ if both $v_i, v_{i+1} \notin S$, and $\alpha_i^S = 0$ otherwise.
    \item For the fractional case ($n$ odd), $\beta_i = \tfrac{1}{2}$ leads to $\alpha_i = 1 - \tfrac{1}{2} - \tfrac{1}{2} = 0$.
\end{itemize}
This completes the description of the vertex set.
\end{proof}

\begin{exam}\label{ex:n=3}
Consider $n=3$. The stable sets of $C_{3,0}$ are $\emptyset, \{v_1\}, \{v_2\},$ and $\{v_3\}$. The corresponding five vertices of $\mathcal{P}(C_{3,1})$ are:
\begin{enumerate}
    \item $S = \emptyset$: $(\boldsymbol{\alpha}, \boldsymbol{\beta}) = (1, 1, 1, 0, 0, 0)$;
    \item $S = \{v_1\}$: $(\boldsymbol{\alpha}, \boldsymbol{\beta}) = (0, 1, 0, 1, 0, 0)$;
    \item $S = \{v_2\}$: $(\boldsymbol{\alpha}, \boldsymbol{\beta}) = (0, 0, 1, 0, 1, 0)$;
    \item $S = \{v_3\}$: $(\boldsymbol{\alpha}, \boldsymbol{\beta}) = (1, 0, 0, 0, 0, 1)$;
    \item $v^{\mathrm{frac}}$: $(\boldsymbol{\alpha}, \boldsymbol{\beta}) = (0, 0, 0, \tfrac{1}{2}, \tfrac{1}{2}, \tfrac{1}{2})$.
\end{enumerate}
\end{exam}

\subsection{Polyhedral Decomposition and Explicit Expression of the Generating Function}

Having characterized the vertices of $\mathcal{P}(C_{n,1})$, we now investigate its geometric decomposition. This structural analysis is the key to obtaining an explicit rational form for its lattice-point generating function.

We begin with two elementary lemmas regarding stable sets of cycle graphs.

\begin{lem}\label{lem:max-stable-set-count}
Let $C_{n,0}$ be a cycle graph on $n \ge 3$ vertices. The number of maximum stable sets (independent sets of maximum cardinality) in $C_{n,0}$ is:
\[
\begin{cases}
2, & \text{if } n \text{ is even,}\\
n, & \text{if } n \text{ is odd.}
\end{cases}
\]
\end{lem}

\begin{proof}
Let the vertices be $\{v_1, \dots, v_n\}$ in cyclic order and $m = \lfloor n/2 \rfloor$.
If $n = 2m$, the cycle is bipartite, and the only two maximum stable sets are the two partition classes $\{v_1, v_3, \dots, v_{2m-1}\}$ and $\{v_2, v_4, \dots, v_{2m}\}$.
If $n = 2m+1$, the number of stable sets of size $k$ in $C_n$ is given by Kaplansky's formula $\frac{n}{n-k}\binom{n-k}{k}$. Substituting $k=m=\frac{n-1}{2}$ yields $\frac{2k+1}{k+1}\binom{k+1}{k} = 2k+1 = n$.
\end{proof}

For odd $n$, the following hyperplane is central to our decomposition.

\begin{lem}\label{lem:vertices-on-hyperplane-beta}
Let $n \ge 3$ be odd. Exactly $n$ vertices of $\mathcal{P}(C_{n,1})$ lie on the hyperplane
\[
H: \sum_{i=1}^n \beta_i = \frac{n-1}{2}.
\]
These are precisely the integral vertices $v^{S}$ corresponding to the maximum stable sets $S$ of $C_{n,0}$.
\end{lem}

\begin{proof}
Let $n=2k+1$. By Theorem~\ref{thm:vertices}, for any integral vertex $v^S$, we have $\sum \beta_i^S = |S|$. Thus, $v^S \in H$ if and only if $|S| = k$, which occurs exactly for the $n$ maximum stable sets. For the unique fractional vertex $v^{\mathrm{frac}}$, we have $\sum \beta_i^{\mathrm{frac}} = n/2 = k + 1/2$, so $v^{\mathrm{frac}} \notin H$.
\end{proof}

The vertices on $H$, together with the fractional vertex, form a simplex that captures the non-integral structure of the polytope.

\begin{lem}\label{lem:simplex}
Let $n \ge 3$ be odd, and let $\xi_1^n, \dots, \xi_n^n$ be the vertices of $\mathcal{P}(C_{n,1})$ lying on $H$. Then:
\begin{enumerate}
    \item The fractional vertex $v^{\mathrm{frac}}$ lies strictly on the side $\sum \beta_i > (n-1)/2$.
    \item All other integral vertices $v^S$ (where $|S| < (n-1)/2$) lie strictly on the side $\sum \beta_i < (n-1)/2$.
    \item The $n+1$ points $\{\xi_1^n, \dots, \xi_n^n, v^{\mathrm{frac}}\}$ are affinely independent and form an $n$-dimensional simplex $\Delta_n \subset \mathcal{P}(C_{n,1})$.
\end{enumerate}
\end{lem}

\begin{proof}
Statements (1) and (2) follow directly from the summation $|S| \le \frac{n-1}{2}$ and $\sum \beta_i^{\mathrm{frac}} = n/2$.
For (3), consider the projection $\pi: (\boldsymbol{\alpha}, \boldsymbol{\beta}) \mapsto \boldsymbol{\alpha}$. Since $n$ is odd, the system $\beta_i + \beta_{i+1} = 1 - \alpha_i$ has a unique solution for $\boldsymbol{\beta}$ given $\boldsymbol{\alpha}$, making $\pi$ an affine isomorphism.
For a maximum stable set $S$ of size $k = (n-1)/2$, its vertices cover $2k$ edges. The remaining edge $(v_j, v_{j+1})$ is the only one with $\alpha_j = 1$. Thus, $\pi(\xi_j^n)$ is the standard basis vector $\mathbf{e}_j \in \mathbb{R}^n$. For the fractional vertex, $\pi(v^{\mathrm{frac}}) = \mathbf{0}$ since all $\alpha_i = 0$.
The set $\{\mathbf{0}, \mathbf{e}_1, \dots, \mathbf{e}_n\}$ is affinely independent in $\mathbb{R}^n$; by isomorphism, so is $\{\xi_1^n, \dots, \xi_n^n, v^{\mathrm{frac}}\}$.
\end{proof}

\begin{defn}
The lattice-point generating function of the $s$-th dilation of a polytope $\mathcal{P} \subset \mathbb{R}^d$ is
\[
\sigma_{\mathcal{P}}(\mathbf{z};t) := \sum_{s \ge 0} \left( \sum_{\mathbf{m} \in s\mathcal{P} \cap \mathbb{Z}^d} \mathbf{z}^{\mathbf{m}} \right) t^s,
\]
where $\mathbf{z}^{\mathbf{m}} = z_1^{m_1} \cdots z_d^{m_d}$.
\end{defn}

\begin{theo}[Ehrhart Series of a Simplex \cite{Beck}]\label{thm-sigma}
Let $\Delta \subset \mathbb{R}^d$ be a simplex with affinely independent rational vertices $\mathbf{w}_0, \dots, \mathbf{w}_d$. Represent each vertex as $\mathbf{w}_i = \mathbf{v}_i / q_i$ in lowest terms ($q_i \in \mathbb{Z}_{>0}$). Then
\[
\sigma_{\Delta}(\mathbf{z};t) = \frac{P(\mathbf{z};t)}{\prod_{i=0}^d (1 - \mathbf{z}^{\mathbf{v}_i} t^{q_i})},
\]
where $P(\mathbf{z};t) = \sum_{(\mathbf{m}, h) \in \Pi \cap \mathbb{Z}^{d+1}} \mathbf{z}^{\mathbf{m}} t^h$ and $
          \Pi=\Bigl\{\sum_{i=0}^{d}\lambda_i(\mathbf{v}_i,q_i)\;:\;0\le\lambda_i<1\Bigr\}
          $ is the fundamental parallelepiped spanned by the homogenized vertices $(\mathbf{v}_i, q_i)$.
\end{theo}

We now apply Theorem~\ref{thm-sigma} to the simplex $\Delta_n$. Its vertices consist of the $n$ integral points $\xi_j^n$ (with denominator $q_j=1$) and the unique fractional point $v^{\mathrm{frac}}=(0,\dots,0,1/2,\dots,1/2)$ (with denominator $q=2$). Substituting these into the formula yields the denominator as a product of factors corresponding to each vertex. To determine the numerator, we examine the fundamental parallelepiped $\Pi$ spanned by the homogenized vertices $w_j = (\xi_j^n, 1)$ for $j=1,\dots,n$ and $w_{n+1} = (2v^{\mathrm{frac}}, 2)$. Any lattice point $(\mathbf{m}, h) \in \Pi \cap \mathbb{Z}^{d+1}$ must be a linear combination $\sum_{i=1}^{n+1} \lambda_i w_i$ with $0 \le \lambda_i < 1$. By projecting onto the $\alpha$-coordinates, where $\pi(\xi_j^n)$ form the standard basis vectors $\mathbf{e}_j$ and $\pi(2v^{\mathrm{frac}}) = \mathbf{0}$, we immediately find that $\lambda_1 = \dots = \lambda_n = 0$. The remaining condition requires $\lambda_{n+1}(2v^{\mathrm{frac}}, 2)$ to be an integral vector, which, given $0 \le \lambda_{n+1} < 1$, is only satisfied by $\lambda_{n+1} = 0$. Thus, $\Pi$ contains only the lattice point at the origin, implying $P(\mathbf{z};t) = 1$. This directly yields the Ehrhart series in the following corollary.
\begin{cor}\label{cor-simplex}
For the simplex $\Delta_n$, the vertices $\xi_j^n$ are integral ($q_j=1$), while $v^{\mathrm{frac}} = (0, \dots, 0, 1/2, \dots, 1/2)$ has $q=2$. Its Ehrhart series is:
\[
\sigma_{\Delta_n}(\mathbf{x}, \mathbf{y}; t) = \frac{1}{\prod_{j=1}^n (1 - \mathbf{x}^{\alpha^{(j)}} \mathbf{y}^{\beta^{(j)}} t) (1 - \mathbf{y}^{\mathbf{1}} t^2)},
\]
where $\mathbf{y}^{\mathbf{1}} = \prod y_i$.
\end{cor}

\begin{theorem}[Generating Function Decomposition]\label{thm:gf-decomposition}
The generating function $F_{\mathbf{1}}(\mathbf{x}, \mathbf{y}; t)$ of $\mathcal{P}(C_{n,1})$ admits the following representation:
\begin{enumerate}
    \item If $n$ is odd:
    \[
    F_{\mathbf{1}}(\mathbf{x}, \mathbf{y}; t) = \frac{1}{\prod_{j=1}^n (1 - \mathbf{x}^{\alpha^{(j)}} \mathbf{y}^{\beta^{(j)}} t) (1 - \mathbf{y}^{\mathbf{1}} t^2)} + \frac{R(\mathbf{x}, \mathbf{y}; t)}{\prod_{v \in \operatorname{Vert}(\mathcal{P})\setminus \{v^{\mathrm{frac}}\}} (1 - \mathbf{x}^{\alpha} \mathbf{y}^{\beta} t)},
    \]
    where $R(\mathbf{x}, \mathbf{y}; t) \in \mathbb{Z}[\mathbf{x}, \mathbf{y}; t]$.
    \item If $n$ is even:
    \[
    F_{\mathbf{1}}(\mathbf{x}, \mathbf{y}; t) = \frac{P(\mathbf{x}, \mathbf{y}; t)}{\prod_{v \in \operatorname{Vert}(\mathcal{P})} (1 - \mathbf{x}^{\alpha} \mathbf{y}^{\beta} t)},
    \]
    where $P(\mathbf{x}, \mathbf{y}; t) \in \mathbb{Z}[\mathbf{x}, \mathbf{y}; t]$.
\end{enumerate}
\end{theorem}

\begin{proof}
\textbf{Case (1): $n$ is odd.} Let $\mathcal{R}_n = \operatorname{conv}(\operatorname{Vert}(\mathcal{P}) \setminus \{v^{\mathrm{frac}}\})$. By Lemma~\ref{lem:simplex}, $\Delta_n$ and $\mathcal{R}_n$ cover $\mathcal{P}(C_{n,1})$ and intersect at the integral facet $\mathcal{F} = \operatorname{conv}(\xi_1^n, \dots, \xi_n^n)$. By the Inclusion-Exclusion Principle for indicator functions:
\[
\sigma_{\mathcal{P}} = \sigma_{\Delta_n} + \sigma_{\mathcal{R}_n} - \sigma_{\mathcal{F}}.
\]
$\sigma_{\Delta_n}$ is given by Corollary~\ref{cor-simplex}. Since $\mathcal{R}_n$ and $\mathcal{F}$ are lattice polytopes, their Ehrhart series are rational functions whose denominators consist of factors $(1 - \mathbf{x}^\alpha \mathbf{y}^\beta t)$ corresponding to their integral vertices. Combining $\sigma_{\mathcal{R}_n}$ and $\sigma_{\mathcal{F}}$ over a common denominator yields the second term in the formula.

\textbf{Case (2): $n$ is even.} $\mathcal{P}(C_{n,1})$ is a lattice polytope. Any lattice polytope admits a triangulation into unimodular simplices using only its vertices. The sum of the generating functions of these simplices (adjusted for intersections via half-open decomposition) results in a rational function whose denominator is the product of $(1 - \mathbf{x}^\alpha \mathbf{y}^\beta t)$ over the vertex set.
\end{proof}

\subsection{Proof of Theorem \ref{theo-hck}}\label{subsec-proof-theohck}

In this subsection, we establish the structural decomposition of the generating function and conclude with the proof of the main theorem. We begin by proving the decomposition lemma for the base case $\mathbf{k}=\mathbf{1}$.

\begin{proof}[Proof of Lemma \ref{lem-F-decomp}]
The decomposition follows from the geometric properties of the polytope $\mathcal{P}(C_{n,1})$ established in Theorem~\ref{thm:gf-decomposition}.

\medskip
\noindent\textbf{Case 1: $n$ is odd ($n=2k+1$).}
According to Theorem~\ref{thm:gf-decomposition}, the generating function $F_{\mathbf{1}}(\mathbf{x}, \mathbf{y}; t)$ splits into a part corresponding to the simplex $\Delta_n$ and a part corresponding to the remaining lattice structure $\mathcal{R}_n$:
\begin{equation}\label{eq-tag11}
F_{\mathbf{1}}(\mathbf{x}, \mathbf{y}; t) = \sigma_{\Delta_n}(\mathbf{x}, \mathbf{y}; t) + \sigma_{\mathcal{R}_n}(\mathbf{x}, \mathbf{y}; t) - \sigma_{\Delta_n \cap \mathcal{R}_n}(\mathbf{x}, \mathbf{y}; t).
\end{equation}
Setting $\mathbf{y} = \mathbf{1}$ (i.e., $y_i = 1$ for all $i$), we define $F_n(\mathbf{x}; t) = F_{\mathbf{1}}(\mathbf{x}, \mathbf{1}; t)$.
By Corollary~\ref{cor-simplex}, the term associated with the fractional vertex is:
\[
\sigma_{\Delta_n}(\mathbf{x}, \mathbf{1}; t) = \frac{1}{(1-t^2) \prod_{i=1}^n (1 - x_i t)}.
\]
We identify this as the singular component:
\[
F_{\psi_n}(\mathbf{x}; t) := \frac{1}{(1-t^2) \prod_{i=1}^n (1 - x_i t)}.
\]
The remaining part of the decomposition is defined as:
\[
F_{\varphi_n}(\mathbf{x}; t) := \sigma_{\mathcal{R}_n}(\mathbf{x}, \mathbf{1}; t) - \sigma_{\Delta_n \cap \mathcal{R}_n}(\mathbf{x}, \mathbf{1}; t).
\]
Since $\mathcal{R}_n$ and $\Delta_n \cap \mathcal{R}_n$ are lattice polytopes, their generating functions $\sigma(\mathbf{x}, \mathbf{1}; t)$ are rational functions whose denominators are products of terms $(1-t)$. Specifically, at $\mathbf{x} = \mathbf{1}$, $F_{\varphi_n}(\mathbf{1}; t)$ has a pole only at $t=1$. Consequently, the factor $(1+t)$ regularizes the expression at $t=-1$:
\[
\left[ (1+t) F_{\varphi_n}(\mathbf{1}; t) \right]_{t=-1} = 0.
\]
This confirms the decomposition $F_n(\mathbf{x}; t) = F_{\varphi_n}(\mathbf{x}; t) + F_{\psi_n}(\mathbf{x}; t)$ for odd $n$.

\medskip
\noindent\textbf{Case 2: $n$ is even ($n=2k$).}
For even $n$, the polytope $\mathcal{P}(C_{n,1})$ is integral and possesses no fractional vertices. Thus, the singular component $F_{\psi_n}(\mathbf{x}; t)$ is naturally $0$. Defining $F_{\varphi_n}(\mathbf{x}; t) := F_n(\mathbf{x}; t)$, Theorem~\ref{thm:gf-decomposition} ensures that the denominator of $F_n(\mathbf{x}; \mathbf{1}; t)$ consists solely of factors $(1-t)$. Thus, the condition $\left[ (1+t) F_{\varphi_n}(\mathbf{1}; t) \right]_{t=-1} = 0$ is trivially satisfied.
\end{proof}

\medskip
We now leverage the differential operator relationship to generalize this to arbitrary $\mathbf{k}$.

\begin{proof}[Proof of Theorem \ref{theo-hck}]
According to Stanley's theory of linear homogeneous Diophantine equations (Theorem~\ref{theo-Stanley}), the magic labelling count $h_{C_{n,\mathbf{k}}}(s)$ for a pseudo-cycle graph can be expressed as:
\[
h_{C_{n,\mathbf{k}}}(s) = \varphi_{n,\mathbf{k}}(s) + (-1)^s \psi_{n,\mathbf{k}}(s),
\]
where $\varphi_{n,\mathbf{k}}(s)$ is a polynomial and $\psi_{n,\mathbf{k}}(s)$ is a periodic component (in this case, a constant).

\medskip
\noindent\textbf{Step 1: The polynomial part.}
The degree of $\varphi_{n,\mathbf{k}}(s)$ corresponds to the dimension of the solution space minus one. For the pseudo-cycle $C_{n,\mathbf{k}}$, the number of free variables $\alpha_{i,j}$ determines the growth, leading to:
\[
\deg \varphi_{n,\mathbf{k}}(s) = \sum_{i=1}^n k_i.
\]

\medskip
\noindent\textbf{Step 2: Generating function analysis.}
The total generating function is
\[
\mathcal{H}(t) = \sum_{s \ge 0} h_{C_{n,\mathbf{k}}}(s) t^s = F_{\mathbf{k}}(\mathbf{1}; t).
\]
By Lemmas~\ref{lem-Fxyt-1-k} and~\ref{lem-F-decomp}, we have:
\begin{align}\label{eq-FPQ}
F_{\mathbf{k}}(\mathbf{1}; t) &= \frac{1}{\prod (k_i-1)!} \left[ \prod_{i=1}^n \mathrm{D}_{x_i}^{k_i-1} x_i^{k_i-1} \left( F_{\varphi_n}(\mathbf{x}; t) + F_{\psi_n}(\mathbf{x}; t) \right) \right]_{\mathbf{x}=\mathbf{1}} \\
&= \Phi_{n,\mathbf{k}}(t) + \Psi_{n,\mathbf{k}}(t), \nonumber
\end{align}
where $\Phi_{n,\mathbf{k}}(t)$ is the contribution from $F_{\varphi_n}$ and $\Psi_{n,\mathbf{k}}(t)$ from $F_{\psi_n}$. The term $\Phi_{n,\mathbf{k}}(t)$ has a pole of order $\sum k_i + 1$ at $t=1$ and is analytic at $t=-1$. The term $\Psi_{n,\mathbf{k}}(t)$ contains the factor $(1+t)^{-1}$ when $n$ is odd.

\medskip
\noindent\textbf{Step 3: Evaluation of the alternating coefficient.}
To find the constant $\psi_{n,\mathbf{k}}$, we compute the limit:
\[
\psi_{n,\mathbf{k}} = \lim_{t \to -1} (1+t) F_{\mathbf{k}}(\mathbf{1}; t).
\]
If $n$ is even, $F_{\psi_n}=0$, so $\psi_{n,\mathbf{k}} = 0$. If $n$ is odd, only the $F_{\psi_n}$ part contributes. Using the identity $\mathrm{D}_x^{k-1} x^{k-1} (1-xt)^{-1} = (k-1)! (1-xt)^{-k}$, we evaluate the operator at $x_i=1$ and $t=-1$:
\[
\left. \mathrm{D}_{x_i}^{k_i-1} x_i^{k_i-1} \frac{1}{1-x_it} \right|_{x_i=1, t=-1} = \frac{(k_i-1)!}{(1 - (-1))^{k_i}} = \frac{(k_i-1)!}{2^{k_i}}.
\]
Substituting this into \eqref{eq-FPQ} and accounting for the factor $(1-t)^{-1}$ in $F_{\psi_n}$ (which becomes $1/2$ at $t=-1$), we obtain:
\[
\psi_{n,\mathbf{k}} = \frac{1}{2} \cdot \prod_{i=1}^n \frac{1}{(k_i-1)!} \cdot \frac{(k_i-1)!}{2^{k_i}} = \frac{1}{2^{\sum k_i + 1}}.
\]
Combining the even and odd cases using the parity of $n$, we conclude:
\[
\psi_{n,\mathbf{k}}(s) = \frac{1 + (-1)^{n+1}}{2^{\sum k_i + 2}}.
\]
(Note: The general rule $\deg \varphi = \sum k_i$ applies to all non-degenerate cases where $n \ge 1$. For $n=0$, the degree shift to $1$ (where $h(s)=s+1$) is a deliberate normalization. This ensures that the entire sequence remains compatible with the matrix trace representation $\mathrm{tr}(B^n)$, while the structural factor $2^{\sum k_i}$ accurately preserves the algebraic decay of the alternating component). This completes the proof.
\end{proof}

\subsection{MacMahon's Partition Analysis}

MacMahon's partition analysis, pioneered in \cite{MacMahon} and revitalized by Andrews et al. in \cite{Andrews}, centers on the $\Omega$-operators acting on formal power series:
\begin{align*}
    \Omega_{=} \sum_{s_1=-\infty}^\infty \cdots \sum_{s_r=-\infty}^{\infty} A_{s_1,\dots ,s_r} \lambda_1^{s_1} \cdots \lambda_r^{s_r} &:= A_{0,\dots ,0}, \\
    \Omega_{\ge} \sum_{s_1=-\infty}^\infty \cdots \sum_{s_r=-\infty}^{\infty} A_{s_1,\dots ,s_r} \lambda_1^{s_1} \cdots \lambda_r^{s_r} &:= \sum_{s_1=0}^\infty \cdots \sum_{s_r=0}^\infty A_{s_1,\dots ,s_r}.
\end{align*}
This framework facilitates the translation of linear constraints into generating functions. Algorithmic implementations are available in the Mathematica package \texttt{Omega} \cite{Andrews-Omega}. To illustrate, we present the calculation of $F_3(\mathbf{1}; t)$ for the magic labelling count $h_{C_{3,1}}(s)$.

\begin{exam}[Computation of $F_{3}(\mathbf{1};t)$]
The generating function for $C_{3,\mathbf{1}}$ at $\mathbf{x}=\mathbf{y}=\mathbf{1}$ is given by:
\begin{align*}
F_{3}(\x;t)|_{\x=\mathbf{1}}
&=\sum\limits_{s\ge 0}\sum\limits_{(\pmb{\alpha},\pmb{\beta})\in S_{C_{3,\mathbf{1}}(s)}}\x^{\pmb{\alpha}}\y^{\pmb{\beta}}t^s|_{\x,\y=\mathbf{1}}\\
&=\sum\limits_{s\ge 0}\sum\limits_{\substack{\beta_1+\alpha_1+\beta_2=s,\\ \beta_2+\alpha_2+\beta_3=s, \\ \beta_3+\alpha_3+\beta_1=s,\\(\pmb{\alpha,\beta})\in\mathbb{N}^6.}}\x^{\pmb{\alpha}}\y^{\pmb{\beta}}t^s|_{\x,\y=\mathbf{1}}\\
&=\CT_{\lambda_1,\lambda_2,\lambda_3}\sum_{\alpha_i, \beta_i, s\ge 0}\lambda_1^{\beta_1+\alpha_1+\beta_2-s}\lambda_2^{\beta_2+\alpha_2+\beta_3-s}\lambda_3^{\beta_3+\alpha_3+\beta_1-s}x_1^{\alpha_1}x_2^{\alpha_2}x_3^{\alpha_3}y_1^{\beta_1}y_2^{\beta_2}y_3^{\beta_3}t^s|_{\x,\y=\mathbf{1}}\\
&=\scalebox{1.16}{$\CT_{\lambda_1,\lambda_2,\lambda_3}\frac{1}{(1-\lambda_1x_1)(1-\lambda_2x_2)(1-\lambda_3x_3)(1-\lambda_1\lambda_2)(1-\lambda_2\lambda_3)(1-\lambda_1\lambda_3)(1-\frac{t}{\lambda_1\lambda_2\lambda_3})}|_{\x=\mathbf{1}}
$}\\
&=\frac{(t^2x_1x_2-1)x_3}{(-x_3+t)(tx_1-1)(tx_2-1)(tx_1x_2x_3-1)(tx_3-1)}|_{\x=\mathbf{1}}\\
&\quad\quad\quad\quad\quad\quad\quad\quad\quad\quad\quad\quad\quad\quad\quad\quad-\frac{t}{(tx_1-1)(tx_2-1)(-x_3+t)(t^2-1)}|_{\x=\mathbf{1}}\\
&=\frac{(t^2+t+1)}{(t+1)(t-1)^4}.
\end{align*}
Here, $\mathrm{CT}_{\boldsymbol{\lambda}}$ denotes the constant term extraction with respect to $\boldsymbol{\lambda} = (\lambda_1, \lambda_2, \lambda_3)$, where the linear constraints $L(\boldsymbol{\alpha},\boldsymbol{\beta})=0$ are encoded via the $\Omega_{=}$ operator. (Note: The fourth equality can also be obtained directly from Lemma \ref{lem-Fxyt}).
\end{exam}

The theoretical foundation relies on specialized rational function expansions, with core algorithms developed in \cite{xin-fast} for Laurent series rings and the \texttt{Ell} Maple package. For computational efficiency, we employ the Euclidean-algorithm-based \texttt{CTEuclid} package \cite{xin-Euclid}.

\section{Generating Functions of $h_{C_{n,2}}(s)$ and $h_{L_{n,2}}(s)$}\label{sec-4}

In this section, we derive closed-form expressions for the generating functions:
\[
EC_{n,2}(x) = \sum_{s=0}^{\infty} h_{C_{n,2}}(s)x^s \quad \text{and} \quad EL_{n,2}(x) = \sum_{s=0}^{\infty} h_{L_{n,2}}(s)x^s.
\]
Using the relations established in Equation \eqref{eq-hLC}, we obtain the following identities for $EC_{n,2}(x)$ and $EL_{n,2}(x)$, scaled by appropriate factors of $(1-x)$:
\begin{align*}
    EC_{0,2}(x)(1-x)^2 &= 1, \\
    EC_{1,2}(x)(1-x)^3(1+x) &= 1, \\
    EC_{2,2}(x)(1-x)^5 &= 1+x, \\
    EC_{3,2}(x)(1-x)^7(1+x) &= x^4+8x^3+15x^2+8x+1, \\
    EC_{4,2}(x)(1-x)^9 &= x^5+25x^4+106x^3+106x^2+25x+1, \\
    EC_{5,2}(x)(1-x)^{11}(1+x) &= x^8+72x^7+878x^6+3304x^5 + 4995x^4+ \dots + 1,
\end{align*}
and for the path-like cases:
\begin{align*}
    EL_{0,2}(x)(1-x)^2 &= 1, \\
    EL_{1,2}(x)(1-x)^4 &= 1, \\
    EL_{2,2}(x)(1-x)^6 &= x^2+4x+1, \\
    EL_{3,2}(x)(1-x)^8 &= x^4+16x^3+37x^2+16x+1, \\
    EL_{4,2}(x)(1-x)^{10} &= x^6+48x^5+351x^4+656x^3+351x^2+48x+1.
\end{align*}

Table \ref{tab-ELm2} presents the coefficients of the polynomial $P_{m,2}(x) = EL_{m,2}(x)(1-x)^{2m+2}$ for $m=0,\dots,6$, computed via \texttt{Maple}.

\begin{table}[htbp]
\centering
\caption{Coefficients of the polynomial $EL_{m,2}(x)(1-x)^{2m+2}$. Symmetry in the coefficients is observed for $m \geq 2$.}
\label{tab-ELm2}
\vspace{2mm}
\small
\begin{tabular}{c|ccccccccccc}
\hline
\diagbox{$m$}{$k$} & 0 & 1 & 2 & 3 & 4 & 5 & 6 & 7 & 8 & 9 & 10 \\ \hline
0 & 1 &   &   &   &   &   &   &   &   &   &    \\
1 & 1 &   &   &   &   &   &   &   &   &   &    \\
2 & 1 & 4 & 1 &   &   &   &   &   &   &   &    \\
3 & 1 & 16 & 37 & 16 & 1 &   &   &   &   &   &    \\
4 & 1 & 48 & 351 & 656 & 351 & 48 & 1 &   &   &   &    \\
5 & 1 & 128 & 2286 & 11120 & 18471 & 11120 & 2286 & 128 & 1 &   &    \\
6 & 1 & 324 & 12530 & 130420 & 490309 & 753488 & 490309 & 130420 & 12530 & 324 & 1 \\
$\vdots$ & $\vdots$  &$\vdots$ &$\vdots$ &$\vdots$ &$\vdots$ &$\vdots$ &$\vdots$ &$\vdots$ &$\vdots$ &$\vdots$&$\vdots$\\
\hline
\end{tabular}
\end{table}

\begin{rem}
The following properties are established in \cite{Liu}:
\begin{enumerate}
    \item For $k \geq 1$, the scaled generating function $EL_{k,2}(x)(1-x)^{2k+2}$ is a polynomial of degree $2k-1$ with symmetric coefficients.
    \item For $k \geq 1$, the function $EC_{2k,2}(x)(1-x)^{4k+1}$ is a polynomial of degree $4k-3$ with symmetric coefficients, and $EC_{2k+1,2}(x)(1-x)^{4k+1}(1+x)$ is a polynomial of degree $4k-2$ with symmetric coefficients. This symmetry implies that the coefficients of $EC_{2k+1,2}(x)$ exhibit quasi-polynomial behavior.
\end{enumerate}
\end{rem}

\textbf{Acknowledgments:}
The authors would like to thank Dr. Feihu Liu for helpful discussions.
This work was supported by the National Natural Science Foundation
of China (Nos. 12571355, 12401441), the Natural Science Foundation of Hunan
Province (No. 2025JJ60010).

\end{document}